\documentclass[a4paper,12pt]{amsart}
\usepackage{amsfonts,amsthm,amssymb,amsmath,amscd}
\usepackage[colorlinks=true,linkcolor=blue,citecolor=blue]{hyperref}
\usepackage{amsmath}
\usepackage{tikz-cd}

\usepackage{tabularx}
\usepackage{tcolorbox}
\usepackage[english]{babel}
\usepackage{amssymb,amsmath}
\usepackage{xcolor}

\newtheorem{Theo}{Theorem}[section]
\newtheorem{Prop}[Theo]{Proposition}
\newtheorem{Coro}[Theo]{Corollary}
\newtheorem{Lemm}[Theo]{Lemma}

\newtheorem{Rema}[Theo]{Remark}

\newcommand{\T}{\mathbb{T}}

\newcommand{\D}{\mathbb{D}}
\newcommand{\C}{\mathbb{C}}
\newcommand{\Z}{\mathbb{Z}}

\DeclareMathOperator{\re}{Re}

\def\N{\mathbb{ N}}
\def\R{\mathbb{ R}}

\begin{document}
\title[Riesz summability on boundary lines]{Riesz summability on boundary lines\\ of holomorphic functions of finite order generated by Dirichlet series}

\author[Defant]{Andreas Defant}
\address[]{Andreas Defant\newline  Institut f\"{u}r Mathematik,\newline Carl von Ossietzky Universit\"at,\newline
26111 Oldenburg, Germany.
}
\email{defant@mathematik.uni-oldenburg.de}

\author[Schoolmann]{Ingo Schoolmann}
\address[]{Ingo Schoolmann\newline  Institut f\"{u}r Mathematik,\newline Carl von Ossietzky Universit\"at,\newline
26111 Oldenburg, Germany.
}
\email{ingo.schoolmann@uni-oldenburg.de}

\maketitle

\begin{abstract}
A particular consequence of the famous  Carleson-Hunt theorem is that the Taylor series expansions of
bounded holomorphic functions on the open unit disk  converge almost everywhere on the boundary, whereas on single points the convergence may fail. In contrast, Bayart,  Konyagin, and Queff\'elec constructed an example of  an ordinary Dirichlet series $\sum a_n n^{-s}$, which  on the open right half-plane $[\re >0]$  converges pointwise
to  a bounded, holomorphic function -- but diverges at each point of the imaginary line, although its limit function extends continuously to the closed right half plane. Inspired by a result of M.~Riesz, we study the boundary behavior of holomorphic functions $f$ on the right half-plane
which for some $\ell \ge 0$ satisfy  the growth condition $|f(s)| = O((1 + |s|)^\ell)$ and are generated by some
Riesz germ, i.e., there is a frequency $\lambda = (\lambda_n)$  and a  $\lambda$-Dirichlet series $\sum a_n e^{-\lambda_n s}$ such  that on some open subset of $[\re >0]$  and for some $m \ge 0$ the function $f$ coincides
with the pointwise  limit (as $x \to \infty$)  of  so-called
$(\lambda,m)$-Riesz means
$\sum_{\lambda_n < x} a_n e^{-\lambda_n s}\big( 1-\frac{\lambda_n}{x}\big)^m ,\,x >0\,.$
 Our main results present criteria for pointwise and uniform Riesz summability  of such functions on the boundary line $[\re =0]$, which include conditions that are motivated by classics like the Dini-test or the principle of localization.
\end{abstract}

\tableofcontents

\noindent
\renewcommand{\thefootnote}{\fnsymbol{footnote}}
 \footnotetext{}

\section{Introduction}

A  $\lambda$-Dirichlet series is a series of the form $D=\sum a_{n}(D) e^{-\lambda_{n}s}$, where $(a_{n}(D))$ is a sequence of  complex coefficients (called Dirichlet coefficients), $\lambda=(\lambda_{n})$ a strictly increasing, non-negative real sequence (called frequency), and $s$ a complex variable. A fundamental property states that, whenever $D$ converges at some complex number  $s=\sigma +i\tau$, it converges on the open half plane $[Re>\sigma]$, where its limit defines a holomorphic  function.

To recall two prominent examples observe that the choice $\lambda=(\log n)$ leads to ordinary Dirichlet series $\sum a_{n} n^{-s}$, whereas the choice $\lambda=(n)$ after the substitution $z=e^{-s}$ generates   power series $\sum a_{n} z^{n}$ in one variable.

Generally speaking, fixing a frequency $\lambda$, there are a couple of natural classes of holomorphic functions $f$ on $[\re > 0]$ which are uniquely  assigned to formal $\lambda$-Dirichlet series $D=\sum a_{n}(D) e^{-\lambda_{n}s}$, and then (still in vague terms) an often  highly involved question  is,

\begin{itemize}
  \item
  whether each such function $f$ on $[\re > 0]$ has a pointwise representation  under an appropriate method of summation
of the series $D$,
\item
and if yes,  whether this  representation even extends to (all or a subset of points of) the boundary line $[\re = 0]$.
\end{itemize}
Let us explain this   more precisely.
The results from \cite{DefantSchoolmann6,HardyRiesz}  motivate the following definition.

Given a frequency  $\lambda = (\lambda_n)$  and a holomorphic function $f:~[\re > 0] \to \mathbb{C}$, we call a $\lambda$-Dirichlet series $D=\sum a_{n}(D)e^{-\lambda_{n}s}$ a $\lambda$-Riesz germ of
$f$,
whenever  on some open subset of $U \subset [\re >0]$  and for some $m \ge 0$
\[
f(s) = \lim_{x \to \infty} R_x^{\lambda,m}(D)(s)\,, \,\,\,\, s \in U\,,
  \]
where the $(\lambda,m)$-Riesz means of $D$ in $s \in \mathbb{C}$ are defined by
\[
R_x^{\lambda,m}(D)(s) = \sum_{\lambda_{n} < x} a_n(D)  e^{-\lambda_n s} \Big(  1- \frac{\lambda_n}{x}  \Big)^m\,,\,\,\,x>0.
\]
In \cite[Corollary~2.15]{DefantSchoolmann6} it is proved that  $\lambda$-Riesz germs $D$ of $f$, whenever they  exist,  are unique, and as a consequence we in this case may assign to every such  $f$ the  unique sequence
\begin{equation*}
(a_n(f))_{n} = (a_n(D))_n\,,
\end{equation*}
which we  call  the 'sequence of Bohr coefficients of~$f$'.

    Consequently, given a holomorphic function  $f:~[\re > 0] \to \mathbb{C}$  generated by  the  $\lambda$-Riesz germ
  $D$, we may define the $x$th Riesz mean of order $k \geq 0$
  of $f $ in $s \in \mathbb{C}$  by
  \[
  R_x^{\lambda,k}(f)(s) =\sum_{\lambda_{n} < x} a_n(f) e^{-\lambda_n s} \Big(  1- \frac{\lambda_n}{x}  \Big)^k\,,
  \]
and a natural question then is to which extend these Riesz means 'reproduce' the function itself.

Within this setting a more precise formulation of the above questions reads as follows: Given a frequency $\lambda$ and
a holomorphic function  $f:~[\re > 0] \to \mathbb{C}$  generated by  the  $\lambda$-Riesz germ
  $D$,

\begin{itemize}
\item
  is there any $k \ge 0$ such that   $f$ on $[\re > 0]$ is pointwise $(\lambda, k)$-Riesz summable on $[\re >0]$,
\item
and if yes, to which extend does this approximation transfer to the boundary line $[\re = 0]$?
\end{itemize}
In the rest of this introduction we want to indicate that the first part of this question is fairly well understood,
and why we hence are going to concentrate on the second part.

\subsection{Classics} \label{A}
Let us illustrate all this, recalling some classics for the power series case $\lambda=(n)$.
Each  function  $f$ from the Hardy space $H_{\infty}(\mathbb{D})$ of all bounded and holomorphic functions on the open complex unit ball $\D$ determines  its (formal) Taylor series $P(z)=\sum \frac{\partial^{n} f(0)}{n!} z^{n},\, z \in \mathbb{C}$  (i.e. the  $(n)$-Dirichlet series $D=\sum a_n(D) e^{-ns}$  with  $a_n(D) = \frac{\partial^{n} f(0)}{n!}$ after the substitution $z = e^{-s}$), and moreover, the function $f$ is represented by its Taylor series in the sense that for every $z\in \mathbb{D}$
\begin{equation}\label{uno}
  f(z)=\sum_{n=1}^{\infty}\frac{\partial^{n} f(0)}{n!}z^{n}.
\end{equation}
Since by Fatou's theorem the radial limits of $f$, i.e.
\begin{equation*}
f^{*}(t)=\lim_{r\to 1} f(re^{it}),
\end{equation*}
exist for almost all $t\in [0,2\pi[$, one may ask, if $P$ converges almost everywhere on the boundary $\T$ and if in this case its pointwise limit coincides with $f^{*}$. A consequence of the famous Carleson-Hunt theorem indeed shows that for almost every $t\in [0,2\pi[$
\begin{equation*}
f^{*}(t)=\sum_{n=1}^{\infty} \frac{\partial^{n} f(0)}{n!} e^{-itn}\,.
\end{equation*}
 On the other hand, it is well-known that there exists a function  $f\in H_{\infty}(\D)$, that is  uniformly continuous on $\D$, although  its Taylor series diverges at certain points of the boundary $\T=\{z\in \C: |z|=1\}$.

Several criteria for pointwise convergence of the Taylor series from \eqref{uno} on the boundary are known. Having in mind that $H_{\infty}(\D)$ via $f\mapsto f^{*}$ is  isometrically isomorphic to  $H_{\infty}(\T)$, preserving Taylor coefficients $(\frac{\partial^{n} f(0)}{n!})$ and Fourier  coefficients~$(\widehat{f^{\ast}}(n))$,  i.e.
\begin{equation}\label{aus}
  H_{\infty}(\D) \,= \,H_{\infty}(\T)\,,
\end{equation}
the classical Dini test (see e.g. \cite[p.~53]{Katznelson}) states  that
for  $f\in H_{\infty}(\D) $ and $z_{0}\in \T$
\begin{equation} \label{equalclassic}
f^\ast(z_{0})=\sum_{n= 0}^{\infty} \frac{\partial^{n} f(0)}{n!} z_{0}^{n},
\end{equation}
whenever
\begin{equation*}
\int_{\T} \Big|\frac{f^\ast(z)-f^\ast(z_{0})}{z-z_{0}}\Big|~dz<\infty.
\end{equation*}
Moreover, if $I\subset \T$ is an open set such that $f^\ast(z)=0$ for all $z\in I$, then
\begin{equation} \label{localone}
\sum_{n=0}^{\infty}  \frac{\partial^{n} f(0)}{n!} z^{n}=0,~ z\in I,
\end{equation}
a fact  known as the principle of localization  (see \cite[p.~54]{Katznelson}).

\subsection{Ordinary Dirichlet series - a counter example} \label{B}
The situation changes dramatically, if we jump from the frequency $\lambda = (n)$ to the frequency $\lambda = (\log n)$. Recall that
\begin{equation*}
\mathcal{D}_{\infty}=\mathcal{D}_{\infty}((\log n))
\end{equation*}
denotes the linear space of all ordinary Dirichlet series $D = \sum a_n n^{-s}$
which converge on some half-plane $[\re > \sigma_0]$ and have a limit function $f$ on this half-plane extending to a bounded holomorphic
function on all of  $[\re > 0]$.

A fundamental theorem of Bohr from \cite{BohrStrip} (see e.g. \cite[Theorem~1.5]{defant2018Dirichlet}) shows that every $D \in \mathcal{D}_{\infty}$ in fact converges uniformly on all half-planes $[\re > \varepsilon]\,, \,\,\varepsilon>0$. This fact has many non-trivial consequences -- among others that  $\mathcal{D}_\infty$ endowed with the supremum norm
$\|D\|_\infty = \sup_{\re s >0} |f(s)|$  is a Banach space. For all needed information on ordinary Dirichlet series  see the monographs
\cite{defant2018Dirichlet} and \cite{queffelec2013diophantine}.

It may come as a surprise that in contrast to the  case $\lambda=(n)$,  Bayart,  Konyagin, and Queff\'elec  for the case
$\lambda = (\log n)$ in their article \cite{Bayart} prove the existence of a Dirichlet series $D\in \mathcal{D}_{\infty}$, that diverges at every point on the imaginary line $[\re =0]$ although its limit function $f$ extends continuously to the closed half plane $[ \re \ge 0]$.

Let us explain that this result may be interpreted as a result in infinite dimensional holomorphy as well as a result in harmonic.
Indeed, denote by $H_\infty (B_{c_0})$ the Banach space of all holomorphic (Fr\'echet differentiable) functions
 $g:B_{c_0}\rightarrow X$ endowed with the sup norm, where $B_{c_0}$ denotes  the open unit ball of the Banach space $c_0$
 of all complex null sequences.
  Then  there is a unique isometric linear bijection
 \begin{equation} \label{H=H}
 \mathcal{D}_{\infty}= H_\infty (B_{c_0})\,,\,\,\, D \mapsto g\,,
 \end{equation}
 which preserves Dirichlet coefficients $(a_n(D))$ and monomial coefficients $(\frac{\partial^{\alpha}g(0)}{\alpha!})$ in the sense that
 $a_{n}=\frac{\partial^{\alpha}g(0)}{\alpha!}$, whenever $n=\mathfrak{p}^{\alpha}$; here $\alpha = (\alpha_n)$ stands for a finite multi index with entries from $\mathbb{N}_0$ (we write $\alpha \in \mathbb{N}_0^{(\mathbb{N})}$) and $\mathfrak{p}=(p_{n})$ for the sequence of primes (see \cite{HLS} and  \cite[Theorem~3.8]{defant2018Dirichlet} for details).

 If $D$ and $g$ are associated to each other according to \eqref{H=H}, then by \cite[Theorem~3.8]{defant2018Dirichlet} for every $ s = u + it \in [\re > 0]$
 \begin{equation}\label{booky}
   g(\mathfrak{p}^{-s})=\sum_{n=1}^{\infty} a_{n}n^{-s} = \lim_{x \to \infty } \sum_{\mathfrak{p}^\alpha < x} \frac{\partial^{\alpha}g(0)}{\alpha!} \frac{1}{\mathfrak{p}^{\alpha u}}\frac{1}{\mathfrak{p}^{i\alpha t}}.
 \end{equation}
But for $u = 0$ this is in general not true -- the  Bayart-Konyagin-Queff\'elec example  shows the existence of a function $g \in H_\infty (B_{c_0})$, such that non of the  limits
\begin{equation}\label{BKQ}
  \lim_{x \to \infty } \sum_{\mathfrak{p}^\alpha < x} \frac{\partial^{\alpha}g(0)}{\alpha!} \frac{1}{\mathfrak{p}^{i\alpha t}}
\,,\,\,\, t \in \mathbb{R},
\end{equation}
exists.

For a reformulation of \eqref{BKQ} in terms  of harmonic analysis recall that the  countable product
$\mathbb{T}^\infty$ of the torus $\T = \{z \in \mathbb{C} \colon |z| =1 \}$
forms a compact abelian group, where  the Haar measure is given by the countable product of the  normalized Lebesgue measure.
The Hardy space $H_\infty(\mathbb{T}^\infty)$ is the closed subspace of all $f \in L_\infty(\mathbb{T}^\infty)$
such that the Fourier transforms $\widehat{f}: \widehat{\mathbb{T}^\infty} = \Z^{(\N)}\to \mathbb{C}$ have their supports in  $\N_0^{(\N)}$.
Then (as an analog of \eqref{aus}) there is an isometric isomorphism
\begin{equation} \label{H=H2}
  H_\infty (B_{c_0}) = H_\infty(\mathbb{T}^\infty)\,,\,\,\, g \mapsto f
\end{equation}
preserving monomial coefficients  $(\frac{\partial^{\alpha}g(0)}{\alpha!})$ and Fourier coefficients $(\widehat{f}(\alpha))$ (see e.g. \cite[Theorem~5.1]{defant2018Dirichlet}). Reformulating \eqref{BKQ}, we see that there is a function $f \in  H_\infty(\mathbb{T}^\infty)$
such that
 non of the  limits
\begin{equation}\label{BKQ2}
  \lim_{x \to \infty } \sum_{\mathfrak{p}^\alpha < x} \widehat{f}(\alpha) \frac{1}{\mathfrak{p}^{i\alpha t}}
\,,\,\,\, t \in \mathbb{R}
\end{equation}
exists.

Hedenmalm and Saaksman in \cite{HedenmalmSaksman} proved a Carleson type convergence theorem for $L_2(\mathbb{T}^\infty)$, i.e.
the Fourier series of every function in $L_2(\mathbb{T}^\infty)$ converges almost everywhere.
 A particular consequence of this  deep fact is  that
 there is an alternative  theorem  which sometimes  may compensate  for the loss caused by the Bayart-Konyagin-Queff\'{e}lec example.

 To explain this, recall that all characters $\chi: \mathbb{N} \to \mathbb{T}$, so all completely multiplicative mappings
from $\mathbb{N}$ into $\mathbb{T}$, in a natural way form a compact abelian group (identifying them with $\mathbb{T}^\infty$). Then we know from \cite[Theorem~1.4]{HedenmalmSaksman} (see also \cite[Theorem~2.1]{defant2020variants}),  that for every Dirichlet series $D = \sum a_n n^{-s}$ with $(a_{n})\in \ell_{2}$ (so in particular for Dirichlet series in  $\mathcal{D}_{\infty}$, see e.g. \cite[Corollary 4.11]{defantschoolmann2019Hptheory})
\begin{equation}\label{portugal}
  \sum a_n \chi(n) n^{-it}
\end{equation}
converges for almost all characters
$\chi: \mathbb{N} \to \mathbb{T}$ and almost all $t \in \mathbb{R}$. This is a considerable improvement of an earlier result of  Helson from \cite{Helson3} (see also \cite[Theorem~9]{HelsonBook}).

Supplementing all this,  we are going to show (see Corollary~\ref{APcaseAP}) that the limit function $f$ of any ordinary Dirichlet series
$D = \sum a_n n^{-s}\in \mathcal{D}_\infty $ extends almost everywhere to the imaginary line $i\mathbb{R}$, where it is almost everywhere  $((\log n),k)$-Riesz-summable at any order $k >0$, i.e. for almost all $t \in \mathbb{R}$ and all $k >0$ the $((\log n),k)$-Riesz limit
\begin{equation}\label{goal}
   \lim_{x \to \infty} \sum_{\log n < x} a_n \frac{1}{n^{it}}\Big(  1- \frac{\log n}{x} \Big)^k\,
\end{equation}
exists. Moreover, it will turn out that under certain further analytic assumptions on the limit function $f$ of $D$ on $[\re >0]$ this convergence improves considerably (Section~\ref{unif}).

\subsection{General Dirichlet series and uniform almost periodicity} \label{C}
The results on ordinary Dirichlet series which we just indicated, will be performed within a setting
of general Dirichlet series -- a far more challengeing task.

Fixing a frequency
$\lambda$, an extensive  'modern' study of $\lambda$-Dirichlet series $\sum a_n e^{-\lambda_n s}$ has been started in the recent articles
\cite{Ba20},  \cite{CaDeMaSc_VV}, \cite{defant2018Dirichlet}, \cite{defantschoolmann2019Hptheory}, \cite{defant2020riesz}, \cite{defant2020variants}, \cite{schoolmann2018bohrA},  \cite{schoolmann2018bohr}, and one of the major concepts is the introduction of so-called Hardy spaces $\mathcal{H}_{p}(\lambda)$ of $\lambda$-Dirichlet series
(in Section~\ref{carleson} we repeat the definition).

The particular case $p=\infty$ is of special interest, since then $\mathcal{H}_{\infty}(\lambda)$ may be described in terms of holomorphic functions on the right half-plane,  and in fact
our purposes in this article demand only this case.

Therefore, recall that $H_{\infty}^{\lambda}[\re > 0]$ (as defined in \cite{defant2020riesz}) denotes the linear  space of all holomorphic and bounded functions  $f\colon [\re > 0]\to \C$, which are  almost periodic on all  vertical lines $[\re =\sigma]$
(or equivalently, some  line $[\re =\sigma]$)  and have  Bohr coefficients
  \begin{equation} \label{bohrcoeffintro}
a_{x}(f)=\lim_{T\to \infty} \frac{1}{2T} \int_{-T}^{T} f(\sigma +it) e^{(\sigma+it)x} dt, \quad x \in \mathbb{R}
\end{equation}
  supported in $\{\lambda_{n} \mid n \in \N\}$.  Note that here
   the limits \eqref{bohrcoeffintro}  are independent of  the choice of $\sigma>0$.
   See e.g. \cite[Section~1.5.2.2]{queffelec2013diophantine} for  the definition
  of almost periodic functions on $\mathbb{R}$, in particular \cite[Theorem 1.5.5]{queffelec2013diophantine} for a couple of  important equivalent reformulations
  of its definition.

   Together with the sup norm, taken  on the right half-plane, $H_{\infty}^{\lambda}[\re > 0]$ forms a Banach space, and by \cite[Theorem 2.16]{defant2020riesz} there is a coefficient preserving isometric linear bijection
   identifying the Hardy space $\mathcal{H}_{\infty}(\lambda)$ of $\lambda$-Dirichlet series and $H_{\infty}^{\lambda}[\re > 0]$,
   \begin{equation} \label{periodic-coin}
     \mathcal{H}_{\infty}(\lambda) \,= \,H_{\infty}^{\lambda}[\re > 0]\,.
   \end{equation}
   Moreover, by \cite[Corollary~4.11]{DefantSchoolmann6} we know that for every bounded and  holomorphic function $f: [\re >0] \to \mathbb{C}$ we have
   \begin{equation} \label{aha}
     \text{$f \in H_{\infty}^{\lambda}[\re > 0]$ if and only if $f$ has a $\lambda$-Riesz germ}\,.
   \end{equation}
       Let us consider an important subspace of $H_{\infty}^{\lambda}[\re > 0]$.  By $\mathcal{D}_\infty(\lambda)$ we denote the space of all Dirichlet series $\sum a_n(D) e^{-\lambda_n s}$ which
converge on $[\re > 0]$ and have a bounded limit function $f:[\re > 0] \to \mathbb{C}$. Then $\mathcal{D}_\infty(\lambda)$ is a normed space if we endow it with the supremum norm  $\|f\| = \sup_{\re s >0}|f(s)| $, and by \cite[Corollary~2.17]{defant2020riesz} we may interpret it as an isometric subspace of $H_{\infty}^{\lambda}[\re > 0]$, where the Dirichlet and Bohr coefficients are preserved.

We say that a frequency $\lambda$ satisfies Bohr's theorem whenever
every $\lambda$-Dirichlet series $D = \sum a_n(D) e^{-\lambda_n s}$, which converges on some half-plane and  has a limit function extending to a bounded, holomorphic function to $[\re > 0]$, in fact converges uniformly on all half-planes $[\re > \varepsilon], \, \varepsilon >0$;
in other terms,
\begin{equation*}
f(s)= \lim_{x \to \infty} R_x^{\lambda,0} f(s) = \lim_{x \to \infty}\sum_{\lambda_n < x} a_n(f) e^{-\lambda_n s}
\end{equation*}
uniformly on $[\re > \varepsilon]$ for all $\varepsilon >0$.
As indicated in the preceding section the frequency $\lambda = (\log n)$ satisfies Bohr's theorem.

A delicate question, which came up  in \cite{defant2020variants},  then is whether we have
\begin{equation} \label{D=H}
     \mathcal{D}_{\infty}(\lambda) \,= \,H_{\infty}^{\lambda}[\re > 0]\,,
   \end{equation}
i.e., each $f\in H_{\infty}^{\lambda}[\re > 0]$ is represented by its Dirichlet series in the sense that $f(s)=\sum_{n=1}^{\infty} a_{n} e^{-\lambda_{n}s}$ for all $s\in [\re > 0]$.
A positive answer  is provided by the so-called equivalence theorem  from \cite[Theorem~5.1]{defant2020variants}
stating that \eqref{D=H} holds if and only if $\lambda$ satisfies Bohr's theorem if and only if $\mathcal{D}_\infty(\lambda)$ is a Banach space  if and only if for every $\sigma>0$ there is a constant $C>0$ such that for all complex sequences $(a_{n})$ and $M\in \mathbb{N}$
\begin{equation*}
\sup_{N\le M} \sup_{t\in \R} \big|\sum_{n=1}^{N} a_{n}e^{-it\lambda_{n}}\big|\le Ce^{\sigma \lambda_{M}} \sup_{t\in \R} \big|\sum_{n=1}^{M} a_{n}e^{-it\lambda_{n}}\big|
\end{equation*}
(for the third equivalence see
\cite[Theorem~4.12]{CaDeMaSc_VV}).
Counterexamples of  frequencies $\lambda$,  failing the equality in \eqref{D=H}, are then provided by \cite[Theorem~5.2]{schoolmann2018bohr}, and  concrete sufficient conditions on $\lambda$ were given by Bohr \cite{Bohr2}, Landau~\cite{Landau}, and more recently  Bayart~\cite{Ba20}. These criteria in particular prove that
\begin{equation} \label{DD=HH}
\mathcal{D}_{\infty}((n)) \,= \,H_{\infty}^{(n)}[\re > 0]
\,\,\,\, \text{and}
     \,\,\,\,
     \mathcal{D}_{\infty}((\log n)) \,= \,H_{\infty}^{(\log n)}[\re > 0]\,.
        \end{equation}

More generally than what we  announced  for  the ordinary case in \eqref{goal}, we are going to show that every $f \in H_{\infty}^{\lambda}[\re > 0]$ extends almost everywhere to the imaginary axis, where it for all $k >0$ is $(\lambda,k)$-Riesz summable almost everywhere, i.e. for almost all  $t \in \R$ the horizontal limits
\begin{equation*}
f^{*}(it)= \lim_{\varepsilon\to 0} f(\varepsilon+i\tau)
\end{equation*}
exist, and for all $k >0$
\begin{equation}\label{herbst}
  f^{*}(it) =  \lim_{x \to \infty}R_x^{\lambda,k}(f)(it) =\sum_{\lambda_{n} < x} a_n(f) e^{-i \lambda_n t} \Big(  1- \frac{\lambda_n}{x}  \Big)^k
\end{equation}
    (see Corollary~\ref{APcase}).
    In particular, \eqref{herbst}  holds true for limit functions of Dirichlet series from $\mathcal{D}_\infty(\lambda)$,
 since (as mentioned before) $\mathcal{D}_{\infty}(\lambda)\subset H_{\infty}^{\lambda}[\re >0]$.
Moreover, the convergence in \eqref{herbst} improves considerably  whenever $f$ fulfills certain additional assumptions.

\bigskip

What happens, if we consider unbounded holomorphic functions
on $[\re >0]$?

\subsection{Two theorems of M. Riesz} \label{D}
Indeed, already M. Riesz in his article \cite{Riesz} from 1909 gave a positive answer, stating a sufficient condition for a wider class of holomorphic functions on $[\re >0]$ which are not necessarily bounded. We recall his two beautiful results from \cite[Theorem~41,~42]{HardyRiesz}, which  (here reformulated using  our notions) in fact were  the starting point of our research.

\begin{Theo}\label{Theo41} Let $f\colon [Re>0] \to \C$ be holomorphic with a $\lambda$-Riesz germ.
Assume that there is $\ell\ge 0$ such that
\begin{equation} \label{conditionintroA}
\forall~ \varepsilon>0~ \exists~ C(\varepsilon)>0\colon |f(s)|\le C(\varepsilon)|s|^{\ell}, ~~ s\in [Re>\varepsilon].
\end{equation}
Then for every $k>\ell$ and $s\in [Re>0]$
\begin{equation*}
f(s)=\lim_{x\to \infty}R_{x}^{\lambda,k}(f)(s).
\end{equation*}
\end{Theo}

\begin{Theo} \label{Theo42}
Let $f\colon [Re>0] \to \C$ be holomorphic with a $\lambda$-Riesz germ and $k>\ell\ge 0$. Assume that  $f$ extends continuously to $[Re\ge 0]$ with the exception of finitely many poles $p_{1},\ldots p_{m}$ on $[Re=0]$ of order $< k+1$. If there exist  $C,\tau_{0}>0$ such that for all $s=\sigma +i\tau \in [Re> 0]$ with $|\tau|\ge \tau_{0}$
\begin{equation*}
|f(s)|\le C |s|^{\ell}\,,
\end{equation*}
then for every $i\tau\notin \{p_{1},\ldots, p_{m}\}$ we have
\begin{equation*}
f(i\tau) = \lim_{x\to \infty}R_{x}^{\lambda,k}(f)(i\tau)\,.
\end{equation*}
Moreover, on every closed interval $I\subset [Re=0]\setminus \{p_{1},\ldots, p_{m}\}$  the convergence is uniform.
\end{Theo}

We illustrate the last theorem with a  well-known example (see e.g. \cite{HardyRiesz}). Take the Riemannian Dirichlet series $\sum n^{-s}$, that converges absolutely on $[\re >1]$,  and consider its analytic continuation $\xi$, namely the zeta function, on $[\re > 0]$  given by the formula
\begin{equation} \label{zeta}
(1-2^{1-s}) \xi(s)=\sum_{n=1}^{\infty} (-1)^{n} n^{-s}\,.
\end{equation}
Recall that $\xi$  has a simple pole at $s=1$, and satisfies the estimate
\begin{equation} \label{zetaestimate}
|\xi(1+it)|\le C\log(|t|), ~~ |t|\ge 1.
\end{equation}
After an obvious translation, Theorem \ref{Theo42} is applicable for every $\ell>0$, and so as a consequence
for every $k>0$ and $\tau \in \mathbb{R} \setminus\{0\}$
\begin{equation}\label{zetapole}
  \xi(1+i\tau)=\lim_{x\to \infty}\sum_{\log(n)<x} n^{-(1+i\tau)} \Big(1-\frac{\log(n)}{x}\Big)^{k}\,,
\end{equation}
with uniform convergence on every closed interval $I\subset \mathbb{R} \setminus\{0\}$.

\subsection{New results in a new  setting} \label{E}
Motivated by these two theorems of Riesz we in \cite{DefantSchoolmann6} define, given a frequency $\lambda$
and $\ell \ge 0$, the space
\begin{equation*}
H_{\infty,\ell}^\lambda[\re>0]\,,
\end{equation*}
collecting    all
holomorphic functions $f:~[\re > 0] \to \mathbb{C}$, which are generated by a $\lambda$-Riesz germ and satisfy
the growth condition
\begin{equation}\label{norway}
  \|f\|_{\infty,\ell} = \sup_{\re s > 0} \Big|
\frac{f(s)}{(1+s)^\ell} \Big|< \infty\,.
\end{equation}
That  this in fact leads to a Banach space $(H_{\infty,\ell}^\lambda[\re>0]), \|\cdot\|_{\infty,\ell})$ is a non-trivial fact proved in \cite[Theorem~3.16]{DefantSchoolmann6}. The case $\ell=0$ is of special interest, since then
\begin{equation}\label{l=0}
    H_{\infty}^{\lambda}[Re>0] = H_{\infty,0}^{\lambda}[\re > 0]\,\,\,\, \text{isometrically\,;}\,
  \end{equation}
this  was already remarked in \eqref{aha} within a slightly different context.

In \cite{DefantSchoolmann6} we performed a sort of  structure theory of these Banach spaces -- mainly based on a considerable extension of Theorem~\ref{Theo41}. In fact, most of the results we derive there, are consequences
of the  following approximation theorem from \cite[Theorem~3.7]{DefantSchoolmann6} for functions in $H^{\lambda}_{\infty,\ell}[\re >0]$ in terms of their Riesz means.

\begin{Theo} \label{41}
Let $k >\ell \ge 0$ and $f\in H_{\infty,\ell}^\lambda[\re >0]$. Then
for every $u >0$
\begin{equation*}
\lim_{x\to \infty} R_x^{\lambda,k}(f)
 (u + \pmb{\cdot})= f(u+\pmb{\cdot})  \,\,\,\text{ in }   \,\,\, H_{\infty,\ell}^\lambda[\re >0].
\end{equation*}
In particular,
 for every  $s\in [Re>0]$
\begin{equation*}
f(s)=\lim_{x\to \infty} R_{x}^{\lambda,k}(f)(s).
\end{equation*}
\end{Theo}

The central question we intend to study in this article then is, to which extent functions in $H_{\infty,\ell}^\lambda~[\re >0]$ are Riesz summable on the  imaginary line.

\smallskip

Let us sketch the main results we establish by fixing  a frequency $\lambda$.  We in Remark~\ref{berlin} observe that, whenever the Riesz limit of a Dirichlet series $D$ exists at some point $i\tau_{0} \in i\mathbb{R}$, and so the limit function $g$ of $D$ defines a holomorphic function on $[Re>0]$, then necessarily
\begin{equation*}
\lim_{x\to \infty} R_{x}^{\lambda,k}(D)(i\tau_{0})=\lim_{u\to 0}g(u+i\tau_{0}).
\end{equation*}
Hence, one of the main properties of functions $f\in H_{\infty,\ell}^{\lambda}[\re > 0]$ we take advantages of,  is given by the fact that the horizontal limits
\begin{equation}\label{horizontalintro}
f^{*}(it)=\lim_{u\to 0} f(u+it)
\end{equation}
exist for almost every $t\in \R$ and define a measurable function (see Proposition~\ref{existencehorizontalB} and Corollary~\ref{existencehorizontalA}). The proof needs Fatou's famous theorem on
boundary limits (within Stolz regions) of bounded holomorphic function on the disc $\mathbb{D}$.
For these horizontal limits we then in Theorem~\ref{AE} prove that  for all $k>\ell\ge 0$ and
for almost all $\tau\in \R$
\begin{equation*}
f^{*}(i\tau) = \lim_{x\to \infty}R_{x}^{\lambda,k}(f)(i\tau)\,.
\end{equation*}
In Section~\ref{AppendixB} (see also Corollary \ref{convolution1NEW}) we provide an internal characterization  of the closed subspace of $L_{\infty}(\R)$ of
all horizontal limits $f^{*}$ generated by functions $f\in H_{\infty,\ell}^{\lambda}[\re > 0]$.
The main tool for all this
 is a  far-reaching Perron-type formula for such horizontal limits (Theorem~\ref{rasmus}).
Elaborating these almost everywhere results, we show in Theorem~\ref{mainresult1} that, if $f^{*}$ is continuous on some open interval $I\subset [\re = 0]$, then for all $i\tau \in I$
\begin{equation} \label{Theorem5.1Intro}
\lim_{x\to \infty} R_{x}^{\lambda,k}(f)(i\tau)=f^{*}(i\tau)\,,
\end{equation}
with uniform convergence on every closed subinterval $J\subset I$. Even more, the sequence $(R_{x}^{\lambda,k}(f))_{x>0}$ converges uniformly on all 'flattened cones', which include all rectangles of the form $[0,\sigma] + J$, $\sigma>0$.
In Theorem~\ref{principleofloc2}
we show  a principle of localization: Assuming that $g$ is another function in $H_{\infty,\ell}^{\lambda}[\re > 0]$ such  that $f^{*}=g^{*}$ on some open interval $I\subset [\re =0]$, we prove that for all $i\tau \in I$
\[
\text{$\lim_{x\to \infty}R_{x}^{\lambda,k}(f)(i\tau)$ exists if and only if $\lim_{x\to \infty}R_{x}^{\lambda,k}(g)(i\tau)$ exists\,,}
\]
 and in this case
 \begin{equation*}
 \lim_{x\to \infty}R_{x}^{\lambda,k}(f)(i\tau)=\lim_{x\to \infty}R_{x}^{\lambda,k}(g)(i\tau).
 \end{equation*}
We finish with a Dini test in Theorem~\ref{Dinitype1}:
If for $\tau \in \R$ there is $\delta>0$ such that
 \begin{equation*}
 \int_{-\delta}^{\delta} \frac{|f^{*}(i(y+\tau))-f^{*}(i\tau)|}{|y|^{1+k-\ell}} dy<\infty,
 \end{equation*}
then
 \begin{equation*}
 \lim_{x\to \infty} R_{x}^{\lambda,k}(f)(i\tau)=f^{*}(i\tau).
 \end{equation*}

We note that Theorem~\ref{Theo42} in \cite{HardyRiesz} is stated without proof.
Since the original article \cite{Riesz}  of M. Riesz  from 1909 is not easily accessible, we at the end of our article (Section \ref{Appendix2}) provide a full proof of Theorem~\ref{Theo42}.

\smallskip

Eventually, we comment how Theorem~\ref{Theo42} and the main contributions of this article are related to each other.
First, observe that  a function $f\in H_{\infty,\ell}^{\lambda}[\re > 0]$, which can be  continuously extended to all of $[\re \ge  0]$
with the exception of a finite number of points on the boundary line $[\re =  0]$,
never can have  poles at these points (see again \eqref{norway}). So for example the result from \eqref{zetapole}, being a consequence of Theorem~\ref{Theo42}, can not be derived from \eqref{Theorem5.1Intro} (Theorem~\ref{mainresult1}). On the other hand, we hope to convince our reader that focusing on  functions from $H_{\infty,\ell}^{\lambda}[\re >0]$, leads to far more  knowledge which can not be reached under the
restrictions assumed in  Theorem~\ref{Theo42}.

\section{Horizontal limits} \label{basics}

As mentioned in the introduction for $f\in H_{\infty,\ell}^{\lambda}[\re > 0]$ we define the measurable function
  \begin{equation*}
f^{*}: i \mathbb{R} \to \mathbb{C} \,,\,\,\,\,  f^{*}(it)=
\begin{cases} \,\,\lim_{\varepsilon\to 0} f(\varepsilon+it) & \text{ the limit exists }
\\[2ex]
\,\, \,0& \text{ else\,, }
\end{cases}
\end{equation*}
and call it the horizontal limit function of $f$. The purpose of this section is to ensure that this definition indeed is reasonable (see Corollary \ref{existencehorizontalA}). This fact  is based on the following seemingly well-known consequence of
Fatou's theorem on non-radial limits of holomorphic functions on the open unit disc $\mathbb{D}$ (see e.g.  \cite[Lemma 11.22]{defant2018Dirichlet}):

Given a  bounded and holomorphic function  $f: [\re > 0] \to \mathbb{C}$, for almost every $t \in \R$ the horizontal limit
\begin{equation*}
f^{*}(it)= \lim_{\varepsilon\to 0} f(\varepsilon+i t)
\end{equation*}
exists.

In fact, we in the following need a variant of this result, namely the following improvement.

\begin{Prop}\label{existencehorizontalB} Let $f: [\re > 0] \to \mathbb{C}$ be bounded and holomorphic. Then there
is a null set $E$ in $i\mathbb{R}$  such that for all $t \in i\mathbb{R} \setminus E$ and all $y \in \mathbb{R}$ we have that
\begin{equation} \label{weber1}
f^{*}(it) =\lim_{\varepsilon\to 0} f( \varepsilon +i\varepsilon y+it).
\end{equation}
exists.
\end{Prop}

In order to verify   \eqref{weber1} we have to take a deeper look into the proof of \cite[Lemma 11.22]{defant2018Dirichlet},
which basically relies on a classical  improvement of Fatou's theorem showing that bounded holomorphic functions on $\mathbb{D}$ not only have radial limits almost everywhere -- but even boundary limits almost everywhere within
so-called Stolz regions.

To do this, let  $\varphi$ be the Cayley transformation, i.e.
\begin{equation*}
\varphi\colon \overline{\mathbb{D}} \setminus \{1\} \to [\re \ge  0]\,,\,\, \,\, \varphi(z)=\frac{1+z}{1-z}\,,
\end{equation*}
with its inverse
\begin{equation*}
\varphi^{-1}\colon [\re \ge  0] \to  \overline{\mathbb{D}} \setminus \{1\}  \,,\,\,\,\, \varphi^{-1}(s)=\frac{s-1}{s+1}\,.
\end{equation*}
Fix some bounded and holomorphic function $f: [\re > 0] \to \mathbb{C}$, and
 define  for every  $\alpha > 1$ the set
\[
N(\alpha)= \big\{w \in \mathbb{T} \colon  \lim_{\substack{z\in S(\alpha,w)\\ z\to w}} f(\varphi(z))
\,\,\, \text{does not exist}\,\big\}\,,
\]
 where
 \begin{equation*}
 S(\alpha,w)=\{z \in \mathbb{D} : |z-w|\le \alpha(1-|z|)\}
 \end{equation*}
is the so-called Stolz region with respect to $w$
 and $\alpha$.
 Then by the mentioned variant of  Fatou's theorem (see e.g. \cite[Section 23]{defant2018Dirichlet}) we know that
 $N(\alpha)$ for every $\alpha >1$ is a null set in $\mathbb{T}$\,. Moreover, we have that
  \begin{equation}\label{ok}
   \text{$N(\alpha) \subset N(\beta)$ for every choice of $1 < \alpha < \beta$},
 \end{equation}
since  $S(\alpha,w) \subset S(\beta,w)$.

\begin{proof}[Proof of Proposition~\ref{existencehorizontalB}] We first show that for every  $y \in \mathbb{R}$, every $\alpha > |1+iy|$,
and every $ t \in  \mathbb{R}$
\begin{equation}\label{stolz}
\exists~ \varepsilon_{0}~ \forall~ \varepsilon < \varepsilon_{0}\colon\varphi^{-1}(\varepsilon+iy\varepsilon+it )\in S(\alpha, \varphi^{-1}(it )).
\end{equation}
Indeed, for every $\varepsilon > 0$
\begin{align*}
&\frac{|\varphi^{-1}(\varepsilon+iy\varepsilon+it )-\varphi^{-1}(it )|}{1-|\varphi^{-1}(\varepsilon+iy\varepsilon+it )|}=\frac{|\frac{\varepsilon+iy\varepsilon+it -1}{\varepsilon+iy\varepsilon+it +1}-\frac{it -1}{it +1}|}{1-|\frac{\varepsilon+iy\varepsilon+it -1}{\varepsilon+iy\varepsilon+it +1}|}
\\ &
=\frac{1}{|it +1|}\frac{|(\varepsilon+iy\varepsilon+it -1)(it +1)-(it -1)(\varepsilon+iy\varepsilon+it +1)|}{|\varepsilon+iy\varepsilon+it +1|-|(\varepsilon+iy\varepsilon+it -1)|} \,=:A_{\varepsilon}(t,y)\,,
\end{align*}
and we claim that $\lim_{\varepsilon\to 0} A_{\varepsilon}(t,y)=|1+iy|$.
Calculating  the numerator
\begin{align*}
|(\varepsilon+iy\varepsilon+it -1)(it +1)&-(it -1)(\varepsilon+iy\varepsilon+it +1)|\\ & =|-2it  +2(\varepsilon+iy\varepsilon+it )|=2|\varepsilon+iy\varepsilon|\,,
\end{align*}
 gives
\begin{equation*}
A_{\varepsilon}(t,y)=\frac{2|\varepsilon+iy\varepsilon|}{|it +1|} \frac{1}{|\varepsilon+iy\varepsilon+it +1|-|(\varepsilon+iy\varepsilon+it -1)|}.
\end{equation*}
We extend the fraction and obtain
\begin{align*}
A_{\varepsilon}(t,y)
=\frac{|1+iy|}{2|1+it |} \big( |\varepsilon+iy\varepsilon+it +1|+|(\varepsilon+iy\varepsilon+it -1)|\big),
\end{align*}
since after multiplication
\begin{align*}
&\big(|\varepsilon+iy\varepsilon+it +1|-|(\varepsilon+iy\varepsilon+it -1)|\big)
\\ & \cdot
\big(|\varepsilon+iy\varepsilon+it +1|+|(\varepsilon+iy\varepsilon+it -1)|\big)
 \\ &
 =|\varepsilon+iy\varepsilon+it +1|^{2}-|(\varepsilon+iy\varepsilon+it -1)|^{2}
  \\ &=(1+\varepsilon)^{2}-(1-\varepsilon)^{2}
=1+2\varepsilon+\varepsilon^2-(1-2\varepsilon+\varepsilon^2)=4\varepsilon\,.
\end{align*}
Consequently, $A_{\varepsilon}(t,y)$ tends to
$|1+iy|$
as $\varepsilon  \to 0$, and this  completes the proof of~\eqref{stolz}.
Now define for every $y \in \mathbb{R}$
\[
\alpha_y : = |1+iy| +1\,,
\]
as well as
\[
\text{$E(y) := \varphi(N(\alpha_y)) \subset i\mathbb{R}$ \,\,\,\,and \,\,\,\,$E: = \bigcup_{y \in \mathbb{R}}  E(y)\subset i\mathbb{R}$}\,.
\]
Note that by \eqref{ok} we have
\begin{equation}\label{okok}
   \text{$E(y_1) \subset E(y_2)$ for every choice of $y_1 < y_2$}\,,
 \end{equation}
and we claim that  this fact shows that $E$ is a null set in $i \mathbb{R}$.
Indeed, for every $ y \in \mathbb{R}$ there is  $n \in \mathbb{N}$ such that  $|y| < n$, and hence by \eqref{okok}
\[
E = \bigcup_{y \in \mathbb{R}}  E(y)\subset \bigcup_{n \in \mathbb{N}} E(n)\,.\
\]
But since the latter set is a countable union of null sets, the claim follows.
Now  for every $it \in i\mathbb{R} \setminus E$ we  for all $y \in \mathbb{R}$ have that
\begin{equation}\label{okokok}
  \lim_{\substack{z\in S(\alpha_y, \varphi^{-1}(it))\\ z\to  \varphi^{-1}(it)}} g(\varphi(z)) \,\,\,\, \text{exists\,.}
\end{equation}
By \eqref{stolz}  we know that for all  $it \in i \mathbb{R}$ and all $y \in \mathbb{R}$
\begin{equation*}
  \varphi^{-1}(\varepsilon+i\varepsilon y+it) \in S(\alpha_y, \varphi^{-1}(it))\,,
\end{equation*}
whenever $\varepsilon$ is small enough.
Consequently, we deduce from \eqref{okokok} and the fact that
by continuity
\begin{equation*}
 \lim_{\varepsilon \to 0}\varphi^{-1}(\varepsilon+iy\varepsilon+it )=\varphi^{-1}(it )\,,
\end{equation*}
 that   for every $it \in i\mathbb{R} \setminus E$ and every  $y \in \mathbb{R}$
\begin{align*}
   \lim_{\varepsilon \to 0} f(\varepsilon+i\varepsilon y+it)
     \,\,\,\, \text{exists\,.}
\end{align*}
On the other hand, again by \eqref{stolz}, for all  $t \in i \mathbb{R}$  and all $y \in \mathbb{R}$
\begin{equation*}
  \varphi^{-1}(\varepsilon+it) \in S(\alpha_0, \varphi^{-1}(it)) \subset S(\alpha_y, \varphi^{-1}(it))\,,
\end{equation*}
so that another application of \eqref{okokok} assures  that   for every $it \in i\mathbb{R} \setminus E$ and every  $y \in \mathbb{R}$
\begin{equation*}
  \lim_{\varepsilon \to 0} f(\varepsilon+it) = \lim_{\varepsilon \to 0} f(\varepsilon+i\varepsilon y+it)\,.
\end{equation*}
This finishes the proof.
\end{proof}

Proposition~\ref{existencehorizontalB} easily transfers to functions in $H_{\infty,\ell}^{\lambda}[\re > 0]$, which are not
necessarily bounded.

\begin{Coro}\label{existencehorizontalA} Let $\ell\ge 0$ and $f\in H_{\infty,\ell}^{\lambda}[\re > 0]$. Then for almost every $t \in \R$ the horizontal limit
\begin{equation*}
f^{*}(it)= \lim_{\varepsilon\to 0} f(\varepsilon+i t)
\end{equation*}
exists. More generally,  there
is a null set $E$ in $i\mathbb{R}$  such that for all $t \in \mathbb{R} \setminus E$ and all $y \in \mathbb{R}$ we have that
\begin{equation*} \label{weber}
f^{*}(it) =\lim_{\varepsilon\to 0} f( \varepsilon +i\varepsilon y+it).
\end{equation*}
\end{Coro}

\begin{proof}
  The argument is immediate --  apply Proposition~\ref{existencehorizontalB} to the bounded and holomorphic function
$g(s)=f(s)(1+s)^{-\ell}, \, s \in [\re >0]$.
\end{proof}

\section{Convolution with weighted horizontal limits}

For functions $f\in H^{\lambda}_{\infty,\ell}[\re > 0]$ the following convolution formula in Theorem \ref{convolution1} is crucial for the forthcoming sections. In fact, it is  the central tool to establish a Perron-type representation of Riesz means in terms of the horizontal limit functions $f^{\ast}$ of $f$ (Theorem \ref{rasmus}). Therefore, we recall that for $u>0$ the classical
 Poisson kernel $P_{u}$ is given by
 \begin{equation} \label{defPoissonkernel}
 P_{u}(t)=\frac{1}{\pi}\frac{u}{u^{2}+t^{2}} \colon \mathbb{R} \to \mathbb{R}\,,
 \end{equation}
and satisfies $\|P_{u}\|_{L_{1}(\R)}=1$ with Fourier transform
\begin{equation} \label{fourier}
\text{$\widehat{P_{u}}(x)=e^{-u|x|}$ \,\,for all
$x\in \mathbb{R}$.}
\end{equation}

\begin{Theo} \label{convolution1} Let $\ell\ge 0$ and $f\in H^{\lambda}_{\infty,\ell}[\re > 0]$.
Then for every $u+i\tau\in [\re > 0]$
\begin{equation} \label{isom}
\Big[\frac{f^\ast(i\cdot)}{(1+i\cdot)^{\ell}}*P_{u}\Big](\tau)=\frac{f(u+i\tau)}{(1+u+i\tau)^{\ell}}.
\end{equation}
\end{Theo}

\begin{proof}
 We start showing that for all $\varepsilon,u>0$ and all $\tau \in \mathbb{R}$
\begin{equation}\label{tuesday}
\frac{f(u+\varepsilon+i\tau)}{(1+u+i\tau)^{\ell}}
=
\Big[\frac{f(\varepsilon+i\cdot)}{(1+i\cdot)^{\ell}}*P_{u}\Big](\tau).
\end{equation}
If then $\varepsilon\to 0$, the conclusion follows by continuity, the dominated convergence theorem and the observation that $\frac{f^\ast(i\cdot)}{(1+i\cdot)^{\ell}}\in L_{\infty}(\R)$ with
\begin{equation} \label{contiJohanna}
\|\frac{f^\ast(i\cdot)}{(1+i\cdot)^{\ell}}\|_{\infty}\le \|f\|_{\infty,\ell},
\end{equation}
where the latter is valid, since for fixed and admissible $t\in \R$
\begin{equation*}
|\frac{f^\ast(it)}{(1+it)^{\ell}}| = \lim_{\varepsilon \to 0} |\frac{f(\varepsilon+it)}{(1+it)^{\ell}}| \le \|f\|_{\infty,\ell} \lim_{\varepsilon \to 0} \frac{|1+\varepsilon+it|^{\ell}}{|1+it|^{\ell}} = \|f\|_{\infty,\ell},
\end{equation*}
Note first, that looking at Theorem~\ref{41}, it suffices to check \eqref{tuesday} only for $f(s)=e^{-sx}$ with $x\ge 0$. To do so, we recall (see e.g. \cite[Remark~2.10]{DefantSchoolmann6}) that for all $\ell>0$ and $s\in [\re > 0]$
\begin{equation*}
\frac{\Gamma(\ell)}{s^{\ell}}=\mathcal{L}(t^{\ell-1})(s)=\int_{0}^{\infty} e^{-st} t^{\ell-1} dt,
\end{equation*}
where $\mathcal{L}$ denotes the Laplace transform.
Together with \eqref{fourier} we obtain
\begin{align*}
\Gamma(\ell)\frac{f(u+\varepsilon+i\tau)}{(1+u+i\tau)^{\ell}}
&
=\Gamma(\ell)\frac{e^{-(u+\varepsilon+i\tau)x}}{(1+u+i\tau)^{\ell}}
\\&
=e^{-(u+\varepsilon+i\tau)x}\mathcal{L}(t^{\ell-1})(1+u+i\tau)
\\&
=
\int_{0}^{\infty} e^{-(u+\varepsilon+i\tau)x} e^{-(1+u+i\tau)t} t^{\ell-1} dt
\\&
=e^{-(\varepsilon+i\tau)x}\int_{0}^{\infty} e^{-u(t+x)} e^{-(1+i\tau)t} t^{\ell-1}dt\\&=e^{-(\varepsilon+i\tau)x} \int_{0}^{\infty} \int_{\R} P_{u}(y) e^{iy(t+x)} dy\, e^{-(1+i\tau)t} t^{\ell-1}dt
\\&
=e^{-(\varepsilon+i\tau)x} \int_{\R} P_{u}(y) e^{iyx}\int_{0}^{\infty}e^{-(1+i(\tau-y))t} t^{\ell-1} dt dy\\&=e^{-(\varepsilon+i\tau)x} \int_{\R} P_{u}(y) e^{iyx} \mathcal{L}(t^{\ell-1})(1+i(\tau-y)) dy\\&=e^{-(\varepsilon+i\tau)x} \int_{\R} P_{u}(y) e^{iyx} \frac{\Gamma(\ell)}{(1+i(\tau-y))^{\ell}} dy
\\&
=\Gamma(\ell)\int_{\R} P_{u}(y) \frac{e^{-(\varepsilon+i(\tau-y))x}}{(1+i(\tau-y))^{\ell}} dy
=
\Gamma(\ell)\Big[\frac{f(\varepsilon+i\cdot)}{(1+i\cdot)^{\ell}}*P_{u}\Big](\tau),&
\end{align*}
which finishes the argument dividing both sides by $\Gamma(\ell)$.
\end{proof}

\begin{Coro} \label{convolution1NEW} Let $\ell\ge 0$. Then the mapping
\begin{equation} \label{LinfinitymappingNEW}
T \colon H_{\infty,\ell}^{\lambda}[\re > 0]\hookrightarrow L_{\infty}(\R), ~~f \mapsto\frac{f^\ast(i\cdot)}{(1+i\cdot)^{\ell}}
\end{equation}
defines an isometric embedding.
\end{Coro}
\begin{proof}
As by \eqref{contiJohanna} and Theorem \ref{convolution1} we already know that $T$ is continuous and injective, it remains to show that $T$ is isometric. Indeed, applying \cite[Theorem 3.7]{DefantSchoolmann6}, equation \eqref{isom} and the fact that the convolution of $L_{\infty}(\R)$ and $L_{1}(\R)$ functions is continuous, we obtain
\begin{align*}
\|f\|_{\infty,\ell} & = \lim_{u\to 0} \sup_{\tau \in \R}|\frac{f(u+i\tau)}{(1+u+i\tau)^{\ell}}|  = \lim_{u\to 0} \sup_{\tau \in \R}|\Big[\frac{f^\ast(i\cdot)}{(1+i\cdot)^{\ell}}*P_{u}\Big](\tau)| \\ & = \lim_{u\to 0} \|\Big[\frac{f^\ast(i\cdot)}{(1+i\cdot)^{\ell}}*P_{u}\Big]\|_{L_{\infty}(\R)} \le \sup_{u>0}  \|\Big[\frac{f^\ast(i\cdot)}{(1+i\cdot)^{\ell}}*P_{u}\Big]\|_{L_{\infty}(\R)}\\ & \le \sup_{u>0}  \|\frac{f^\ast(i\cdot)}{(1+i\cdot)^{\ell}}\|_{L_{\infty}(\R)} \|P_{u}\|_{L_{1}(\R)} \le \|\frac{f^\ast(i\cdot)}{(1+i\cdot)^{\ell}}\|_{L_{\infty}(\R)}. \qedhere
\end{align*}
\end{proof}

We refer  to Section \ref{AppendixB} from our   appendix, where  we for the sake of completeness in Theorem~\ref{svwhsv} give an internal description of the range of the operator in  \eqref{LinfinitymappingNEW}. For the  particular  case $\ell=0$ we in Theorem~\ref{casezero} show that a function $g\in L_{\infty}(\R)$ belongs to the image of $T$ if and only if the continuous function $g*P_{u}$ is almost periodic for every $u>0$ and the Bohr coefficients

\begin{equation*}
  a_{x}(g*P_{u})=\lim_{T\to \infty} \frac{1}{2T} \int_{-T}^{T} [g*P_{u}](t) e^{ixt} dt, \quad x\in \R,
\end{equation*}
vanish, whenever $x\notin \{\lambda_{n} \mid n\in \N\}$.

\section{Perron's formula in terms of horizontal limits} \label{preparation}
Given a frequency $\lambda$, some $\ell \ge 0$ and $f\in H_{\infty,\ell}^{\lambda}[\re > 0]$, the aim  of this section is to prove an integral formula for the Riesz means $R_{x}^{\lambda,k}(f)$ in terms of the horizontal limit function $f^\ast$, whenever $k>\ell$. Later we are going to see that this integral
description incorporates most of the information we need for the understanding of Riesz summation on the imaginary line.

The following  Perron-type formula  is an indispensable tool  from \cite[Theorem~3.5]{DefantSchoolmann6},
which in fact up to some point rules the structure theory of the scale of  Banach spaces
$H^{\lambda}_{\infty,\ell}[\re >0]$, $\ell\ge 0$.

\begin{Theo} \label{new-ban}
 Let  $f \in H^{\lambda}_{\infty,\ell}[\re >0]$ and $k >\ell\ge 0$.
 Then for all  $s_{0}\in [\re \ge 0]$, $x>0$ and $c >0$
 \begin{equation*}
 R_{x}^{\lambda,k}(f)(s_{0}) = \frac{\Gamma(1+k)}{2\pi i} x^{-k} \int_{c-i\infty}^{c+i\infty} \frac{f(s+s_{0})}{s^{1+k}} e^{xs} ds.
 \end{equation*}
\end{Theo}

We start modifying this formula. First, observe that regarding summation on the imaginary line by translation it suffices to handle the case $s_{0}=0$. Then, the choice $c=x^{-1}$ leads to
\begin{align*}
R_{x}^{\lambda,k}(f)(0)&=\frac{\Gamma(1+k)}{2\pi i} x^{-k}\int_{x^{-1}-i\infty}^{x^{-1}+i\infty} \frac{f(s)}{s^{1+k}} e^{xs} ds\\ &=\frac{\Gamma(1+k)}{2\pi }  x^{-k} \int_{-\infty}^{\infty} \frac{f(x^{-1}+it)}{(x^{-1}+it)^{1+k}}e^{x(x^{-1}+it)} dt\\ &=\frac{\Gamma(1+k)e}{2\pi } \int_{-\infty}^{\infty} f(x^{-1}+it) e^{ixt} \frac{x}{(1+ixt)^{1+k}} dt.
\end{align*}

We fix this observation.

\begin{Rema}\label{perronformula1} Let $k>\ell \ge 0$ and $f\in H_{\infty,\ell}^{\lambda}[\re > 0]$. Then for all  $x>0$
\begin{equation*}
R_{x}^{\lambda,k}(f)(0)=\int_{-\infty}^{\infty} f(x^{-1}+it) e^{ixt} K_{x}^{k}(t) dt,
\end{equation*}
where
\begin{equation*}
K_{x}^{k}(t)=\frac{\Gamma(1+k) e}{2\pi } \frac{x}{(1+ixt)^{1+k}}\,,\,\, \,t \in \mathbb{R}.
\end{equation*}
\end{Rema}

The functions $(K_{x}^{k})_{x>0}$ are generated by the kernel

\begin{equation}\label{K}
  K^{k}(y)=\frac{\Gamma(1+k)e}{2 \pi} \frac{1}{(1+iy)^{1+k}} \,,\,\, \,y \in \mathbb{R}
\end{equation}
in the sense that $K^{k}_{x}(t)=xK^{k}(xt)$  for all $x>0$ and $t \in \mathbb{R}$.
Moreover, note that, provided  $\lambda_{1}=0$, the function $f=1$ belongs to $H_{\infty,\ell}^{\lambda}[\re > 0]$ with  $R_{x}^{\lambda,k}(f)(0)=1$ for all $x>0$. Hence, by Remark~\ref{perronformula1} ($x=1$) we see that
\begin{equation} \label{Fouriervalue}
\int_{\R} e^{iy} K^{k}(y) dy=1.
\end{equation}

With the  aid of Theorem \ref{convolution1} we now continue the modification of the Perron-type formula from Remark~\ref{perronformula1} . The
following result is the main contribution of this subsection -- Perron's formula in terms of horizontal limits.

\begin{Theo}\label{rasmus}Let $f\in H_{\infty,\ell}^{\lambda}[\re > 0]$ and $k>\ell\ge 0$. Then
\begin{equation*}
R_{x}^{\lambda,k}(f)(0)=\int_{\R}  \frac{f^\ast(iy)}{(1+iy)^{\ell}}  R^{k,\ell}(x,y) dy\,,~\, x>0,
\end{equation*}
where
\begin{equation*}
R^{k,\ell}(x,y)=\int_{\R}P_{x^{-1}}(t-y)e^{itx} (1+x^{-1}+it)^{\ell}K^{k}_{x}(t) dt, ~\,y \in \mathbb{R}.
\end{equation*}
\end{Theo}

\begin{proof}
By Remark~\ref{perronformula1} and the convolution formula from Theorem \ref{convolution1} we have
  \begin{align*}
R_{x}^{\lambda,k}(f)(0)&=\int_{\R} f(x^{-1}+it) e^{ixt} K_{x}^{k}(t) dt
\\&
=\int_{\R} \frac{f(x^{-1}+it)}{(1+x^{-1}+it)^{\ell}} e^{ixt} (1+x^{-1}+it)^{\ell}K_{x}^{k}(t) dt \\&=\int_{\R} \Big(\frac{f^\ast(i\cdot)}{(1+i\cdot)^{\ell}}*P_{x^{-1}}\Big)(t)e^{ixt} (1+x^{-1}+it)^{\ell}K_{x}^{k}(t) dt\\&=
\int_{\R} \int_{\R} \frac{f^\ast(iy)}{(1+iy)^{\ell}} P_{x^{-1}}(t-y)  e^{ixt} (1+x^{-1}+it)^{\ell}K_{x}^{k}(t) dy dt\\ &=\int_{\R}\frac{f^\ast(iy)}{(1+iy)^{\ell}}  \int_{\R} P_{x^{-1}}(t-y)  e^{ixt} (1+x^{-1}+it)^{\ell}K_{x}^{k}(t) dt dy \\ &= \int_{\R}  \frac{f^\ast(iy)}{(1+iy)^{\ell}}  R^{k,\ell}(x,y) dy.
\qedhere
\end{align*}
\end{proof}

In Lemma \ref{Matea} we are going to show how to  control  the $L_{1}(\R)$-norm of the functions $R^{k,\ell}(x,\cdot)$, which will be essential in all our applications of the preceding formula.

\section{Almost everywhere convergence} \label{proofs1}
The completion of the preceding preparations have paved the way for the first  of our four main contributions. As before, the sequence $\lambda=(\lambda_{n})$ always denotes  an arbitrary frequency.

\begin{Theo} \label{AE}  Let $f\in H_{\infty,\ell}^{\lambda}[\re > 0]$ and $k>\ell\ge 0$. Then the Dirichlet series
$\sum a_n(f) e^{-\lambda_n s}$ is   $(\lambda,k)$-Riesz summable  almost everywhere on the imaginary line,  and
for almost all  $\tau\in \R$
\begin{equation*}
\lim_{x\to \infty}R_{x}^{\lambda,k}(f)(i\tau)=f^{*}(i\tau)\,.
\end{equation*}
\end{Theo}

Observe that this result for $k=\ell=0$ fails in the ordinary case $\lambda = (\log n)$.   This follows from
the Bayart-Konyagin Queff\'{e}lec example from Section~\ref{B} and the equality in \eqref{l=0}.

Before we come to the proof of Theorem~\ref{AE} we remark a sort of converse of   this result: If $\sum a_n(f) e^{-\lambda_n s}$ is   $(\lambda,k)$-Riesz summable at some point of the boundary line, then the horizontel limit of $f$ exists  and equals the
$(\lambda,k)$-Riesz sum of $f$ at this point. For the special case
$\lambda = (n)$ and $k=0$ this (after the standard reformulation) is nothing else than Abel's classical convergence theorem for power series.

\begin{Rema} \label{berlin}
  Let $f: [\re > 0] \to \mathbb{C}$ be a holomorphic function with a $\lambda$-Riesz germ and $k \geq 0$. Assume that $f$
  is $(\lambda,k)$-Riesz summable at $0$ with limit $A$, i.e.
  $\lim_{x \to \infty} R_x^{\lambda,k}(f)(0) = A$ exists.
  Then
  \[
  \lim_{\sigma \to 0} f(\sigma) = A\,.
  \]
  \end{Rema}
  \begin{proof}
  We assume (without loss of generality) that $A=0$.
  Note first that by assumption  the $\lambda$-Riesz germ of $f$ converges on $[\re >0]$, and that its limit function
  coincides with $f$. Then we know from \cite[Theorem 2.9]{DefantSchoolmann6}  that for each $\sigma >0$
    \[
    f(\sigma) = \frac{1}{\Gamma(k+1)} \sigma^{1+k} \int_0^\infty  S_t^{\lambda,k}(f)(0)e^{-\sigma t} dt\,.
    \]
   Now fix some $\varepsilon >0$, and choose $\tau_0 >0$ such that for all $t > \tau_0$
   \[
   \big| S_t^{\lambda,k}(f)(0)  \big|\leq \varepsilon t^k\,.
   \]
    Consequently,
    for each $\sigma >0$
    \begin{align*}
           |f(\sigma)|
           &
           \leq
           \frac{1}{\Gamma(k+1)} \sigma^{1+k} \int_0^{\tau_0}  S_t^{\lambda,k}(f)(0)e^{-\sigma t} dt
           +
           \frac{1}{\Gamma(k+1)} \sigma^{1+k} \varepsilon \int_{\tau_0} ^\infty t^k e^{-\sigma t} dt
           \\&
           \leq
           \frac{1}{\Gamma(k+1)} \sigma^{1+k} \int_0^{\tau_0}  S_t^{\lambda,k}(f)(0) dt
           +
             \varepsilon\,.
         \end{align*}
    Since the first term  tends to $0$ whenever $\sigma$ tends to $0$, we obviously obtain the  conclusion.
          \end{proof}

\begin{proof}[Proof of Theorem \ref{AE}]
Applying the substitution $y=tx$ in Remark \ref{perronformula1}, we obtain for all $\tau >0$ and all $x >0$
\begin{equation*}
R_{x}^{\lambda,k}(f)(i\tau) =\int_{-\infty}^{\infty} f(x^{-1}+iyx^{-1}+i\tau)e^{iy} K^{k}(y) dy.
\end{equation*}
Since by Corollary~\ref{existencehorizontalA} for all  $y\in \R$ and almost all $\tau >0$ we have
\begin{equation*}
\lim_{x \to \infty} f(x^{-1}+iyx^{-1}+i\tau)=f^{*}(i\tau),
\end{equation*}
the dominated convergence theorem and the use of \eqref{Fouriervalue} imply
\begin{equation*}
\lim_{x\to \infty} R_{x}^{\lambda,k}(f)(i\tau)=f^{*}(i\tau) \int_{-\infty}^{\infty} e^{iy}K^{k}(y) dy =f^{*}(i\tau).
\end{equation*}
Indeed,  fixing  $\tau>1$ and  $x>1$, we for $y \in \mathbb{R}$ have
\begin{align*}
&
 |f(x^{-1}+iyx^{-1}+i\tau)e^{iy} \frac{1}{(1+iy)^{1+k}}|
 \\&
   \leq    \|f\|_{\infty,\ell}  \frac{|1 + (x^{-1}+iyx^{-1}+i\tau)|^\ell}{|1+iy|^{1+k}}
   \\&
    \leq \|f\|_{\infty,\ell} \, 2^{\max(0,\ell -1)}
   \Big(\frac{|1 +iyx^{-1}|^\ell}{|1+iy|^{1+k}}
   + \frac{|x^{-1}+i\tau)|^\ell}{|1+iy|^{1+k}}\Big)
   \\&
    \leq \|f\|_{\infty,\ell} \, 2^{\max(0,\ell -1)}
   \Big(\frac{1}{|1+iy|^{1+k-\ell}}
   + \frac{|1+i\tau)|^\ell}{|1+iy|^{1+k}}\Big)\,. \qedhere
\end{align*}
\end{proof}

 From
\eqref{l=0} we immediately deduce the case $\ell=0$, which is of special interest.

\begin{Coro} \label{APcase}  Let $f\in H_{\infty}^{\lambda}[\re > 0]$ and $k> 0$. Then for almost every $\tau\in \R$
\begin{equation*}
\lim_{x\to \infty}R_{x}^{\lambda,k}(f)(i\tau)=f^{*}(i\tau)\,.
\end{equation*}
In particular, every  Dirichlet series $D\in \mathcal{D}_{\infty}(\lambda)$ is almost everywhere  $(\lambda,k)$-Riesz summable
on the imaginary line.
\end{Coro}

In view of the Bayart-Konyagin-Queff\'{e}lec counterexample (see Section~\ref{B}), it  seems worthwhile  to mention the special case $\lambda = (\log n)$ separately.

\begin{Coro} \label{APcaseAP}  Let $f\in H_{\infty}^{(\log n)}[\re > 0]$ and $k> 0$. Then for almost every $\tau\in \R$
\begin{equation*}
\lim_{x\to \infty}R_{x}^{(\log n),k}(f)(i\tau)=f^{*}(i\tau)\,.
\end{equation*}
In particular,  if $D = \sum a_n n^{-s} \in \mathcal{D}_{\infty}$ is the Dirichlet series associated to $f$ (see again~\eqref{DD=HH}), then for  almost all $\tau \in \mathbb{R}$
\begin{equation*}
   \lim_{x \to \infty} \sum_{\log n < x} a_n \frac{1}{n^{i\tau}}\Big(  1- \frac{\log n}{x} \Big)^k =  f^{*}(i\tau)  \,.
\end{equation*}
\end{Coro}

\section{A principle of localization} \label{DiniA}

The second main result (after Theorem~\ref{AE}) may be seen as a principle of localization -- compare with what we recalled in \eqref{localone} for the one variable case.

\begin{Theo} \label{principleofloc2}
Let $k>\ell\ge 0$ and $f,g\in H_{\infty,\ell}^{\lambda}[\re > 0]$. Assume that $f^{*}=g^{*}$ on some open interval $I\subset [\re > 0]$. Then, given $i\tau \in I$, the limit $\lim_{x\to \infty}R_{x}^{\lambda,k}(f)(i\tau)$ exists if and only if $\lim_{x\to \infty}R_{x}^{\lambda,k}(g)(i\tau)$ exists, and in this case
\begin{equation*}
\lim_{x\to \infty}R_{x}^{\lambda,k}(f)(i\tau)=\lim_{x\to \infty}R_{x}^{\lambda,k}(g)(i\tau).
\end{equation*}
\end{Theo}

The proof of this principle is given at the end of this section, and it turns out to be a simple consequence of the following independently interesting result. Recall the definition of the kernel functions $R^{k,\ell}(x,\cdot)$, $x>0$, from Theorem~\ref{rasmus}.

\begin{Theo} \label{perronformula22} Let $f\in H_{\infty,\ell}^{\lambda}[\re > 0]$ and $k>\ell\ge 0$. Then for every
 $\delta>0$
\begin{equation*}
\lim_{x\to\infty}\int_{|y|\ge \delta} \frac{f^\ast(iy)}{(1+iy)^{\ell}}R^{k,\ell}(x,y)  dy = 0.
\end{equation*}
\end{Theo}

The proof of Theorem~\ref{perronformula22} requires to control the norm of $R^{k,\ell}(x,\cdot)$, which is provided by the following lemma.

\begin{Lemm} \label{Matea}
Let $k>\ell\ge 0$. Then there is a constant $C(k, \ell) >0$ such that for each $x>1$ and every $y \in \mathbb{R}$

\begin{equation*}
|R^{k,\ell}(x,y)|
\leq C(k, \ell)
\begin{cases}
\frac{x}{|1+iyx|^{1+k-\ell}},
 & k <1,
\\[2ex]
    \frac{x}{|1+iyx|^{2}} + \frac{x}{|1+iyx|^{1+k-\ell}},            &  k \ge 1 \,\,\,\text{and} \,\,\,k-\ell \leq 1,
            \\[2ex]
    \frac{x}{|1+iyx|^{2}},           &  k \ge 1  \,\,\,\text{and} \,\,\, k-\ell \ge 1 \,.
\end{cases}
\end{equation*}
Moreover, for every $\delta>0$
\begin{equation} \label{sunday}
\lim_{x\to \infty} \int_{|y|>\delta} |R^{k,\ell}(x,y)| dy =0.
\end{equation}
\end{Lemm}

Let us first deduce Theorem~\ref{perronformula22} from Lemma \ref{Matea}.

\begin{proof}[Proof of Theorem~\ref{perronformula22}]
The 'moreover-part' of Theorem \ref{convolution1} and \eqref{sunday} imply
\begin{align*}
\lim_{x\to \infty} \Big|\int_{|y|\ge \delta} & \frac{f^\ast(iy)}{(1+iy)^{\ell}} R^{k,\ell}(x,y) dy\Big| \le \|f\|_{\infty,\ell}\lim_{x\to \infty} \int_{|y|\ge \delta} |R^{k,\ell}(x,y)| dy =0,
\end{align*}
the conclusion.
\end{proof}

\begin{proof}[Proof of Lemma \ref{Matea}]
Recall (from Theorem~\ref{rasmus} and \eqref{K}) that for $y \in \mathbb{R}$ and $x >0$
\begin{equation*}
R^{k,\ell}(x,y)= \frac{\Gamma(1+k) e}{2\pi } \int_{\R}P_{x^{-1}}(t-y)e^{itx}  \frac{x(1+x^{-1}+it)^{\ell}}{(1+ixt)^{1+k}} dt\,,
\end{equation*}
Then
\begin{equation*}
|R^{k,\ell}(x,y)| \leq \frac{\Gamma(1+k) e}{2\pi } \int_{\R}P_{x^{-1}}(t-y) \frac{x |1+x^{-1}+it|^{\ell}}{|1+ixt|^{1+k}} dt\,,
\end{equation*}
and  hence for $x >1$ (implying $|x^{-1}+it|\leq |1+ixt|$ for every $t \in \mathbb{R}$) we get
\begin{align*}
   &
   \int_{\R}P_{x^{-1}}(t-y) \Big|\frac{x (1+x^{-1}+it)^{\ell}}{(1+ixt)^{1+k}} dt
   \\&
   \leq
      C(\ell) \bigg(
   \int_{\R}P_{x^{-1}}(t-y) \frac{x }{|1+ixt|^{1+k}}dt
   +
   \int_{\R}P_{x^{-1}}(t-y) \frac{x |x^{-1}+it|^{\ell}}{|1+ixt|^{1+k}}dt
         \bigg)
         \\&
   \leq
      C(\ell) \bigg(
   \int_{\R}P_{x^{-1}}(t-y) \frac{x }{|1+ixt|^{1+k}}dt
   +
   \int_{\R}P_{x^{-1}}(t-y) \frac{x}{|1+ixt|^{1+k-\ell}}dt
         \bigg)\,,
                                  \end{align*}
   where  $C(\ell) = 2^{\max(0,\ell -1)}$. So it remains to  control the last  two integrals.
  We already know from \cite[Lemma~3.4]{defant2020riesz} (with $u=0$ and $v=x^{-1}$) that for all $0<\alpha\le 1$
\begin{equation*}
\int_{\R} \frac{P_{x^{-1}}(t-y)}{|x^{-1}+it|^{1+\alpha}} dt\le \frac{2}{|x^{-1}+iy|^{1+\alpha}}.
\end{equation*}
Now assume  first that $k < 1$, which implies  $k-\ell\le 1$. Then
\begin{align*}
   &
   \int_{\R}P_{x^{-1}}(t-y) \frac{x }{|1+ixt|^{1+k}}dt
   +
   \int_{\R}P_{x^{-1}}(t-y) \frac{x}{|1+ixt|^{1+k-\ell}}dt
   \\&
   =x^{-k} \int_{\R} \frac{P_{x^{-1}}(t-y)}{|x^{-1}+it|^{1+k}} dt
   +
   x^{-k+\ell} \int_{\R} \frac{P_{x^{-1}}(t-y)}{|x^{-1}+it|^{1+k-\ell}} dt
                  \\&
   \leq
         x^{-k} \frac{2}{|x^{-1}+iy|^{1+k}}
   +
   x^{-k+\ell} \frac{2}{|x^{-1}+iy|^{1+k-\ell}}
                 \\&
   =
         2\frac{x}{|1+ixy|^{1+k}}
   +
   2\frac{x}{|1+ixy|^{1+k-\ell}}
   \leq
         4\frac{x}{|1+ixy|^{1+k-\ell}}\,,
        \end{align*}
which proves the first claim. Assume second that $k \ge  1$ and $k-\ell\le 1$. Then
 $ |1+ixt|^{2} \leq |1+ixt|^{1+k}$ for  $t \in \mathbb{R}$, and consequently
\begin{align*}
   &
   \int_{\R}P_{x^{-1}}(t-y) \frac{x }{|1+ixt|^{1+k}}dt
   +
   \int_{\R}P_{x^{-1}}(t-y) \frac{x}{|1+ixt|^{1+k-\ell}}dt
   \\   &
    \leq \int_{\R}P_{x^{-1}}(t-y) \frac{x }{|1+ixt|^{2}}dt
   +
   \int_{\R}P_{x^{-1}}(t-y) \frac{x}{|1+ixt|^{1+k-\ell}}dt
   \\&
   =x^{-2} \int_{\R} \frac{P_{x^{-1}}(t-y)}{|x^{-1}+it|^{2}} dt
   +
   x^{-k+\ell} \int_{\R} \frac{P_{x^{-1}}(t-y)}{|x^{-1}+it|^{1+k-\ell}} dt
                  \\&
   \leq
         x^{-2} \frac{2}{|x^{-1}+iy|^{2}}
   +
   x^{-k+\ell} \frac{2}{|x^{-1}+iy|^{1+k-\ell}}
                 \\&
   =
         \frac{2x}{|1+ixy|^{2}}
   +
   \frac{2x}{|1+ixy|^{1+k-\ell}}\,.
                       \end{align*}
Similarly, we handle the case  $k \ge 1$ and $k-\ell>1$ getting
\begin{align*}
      \int_{\R}P_{x^{-1}}(t-y) \frac{x }{|1+ixt|^{1+k}}dt
   +
   \int_{\R}P_{x^{-1}}(t-y) \frac{x}{|1+ixt|^{1+k-\ell}}dt
   \leq
         \frac{4x}{|1+ixy|^{2}}
   \,.
                       \end{align*}
                      The 'moreover part' follows by substitution. Since all three cases follow the same lines, we only consider the case $k<1$. Then
\begin{align*}
\int_{|y|>\delta} |R^{k,\ell}(x,y)| dy & \le C(k,\ell) \int_{|y|>\delta} \frac{x}{|1+iyx|^{1+k-\ell}} dy\\ &=C(k,\ell) \int_{|y|>\delta x} \frac{1}{|1+it|^{1+k-\ell}} dt,
\end{align*}
            which tends to zero as $x\to \infty$.         This completes the proof.
\end{proof}

Finally, we come back to the proof of Theorem~\ref{principleofloc2}.

\begin{proof}[Proof of Theorem~\ref{principleofloc2}]
Translating, if necessary, we  may assume  that $ 0 \in I$ and $\tau = 0$. We choose some $\delta > 0$  such that
$i[-\delta, \delta] \subset I$. Then  by the Perron-type formula from Theorem~\ref{rasmus}
 we have
\begin{equation}\label{split}
  R_{x}^{\lambda,k}(f)(0)=
  \int_{|y|\ge \delta} \frac{f^\ast(iy)}{(1+iy)^{\ell}}R^{k,\ell}(x,y)  dy
  +
  \int_{|y|\leq \delta} \frac{f^\ast(iy)}{(1+iy)^{\ell}}R^{k,\ell}(x,y)  dy\,.
\end{equation}
Since $f^\ast = g^\ast$ on $i[-\delta, \delta] \subset I$, we then observe that the claim  is an immediate consequence of  Theorem~\ref{perronformula22}.
\end{proof}

\section{Uniform convergence} \label{unif}

We come to our third main result, which among others recovers  Theorem \ref{Theo42} for functions in $H_{\infty,\ell}^{\lambda}[\re > 0], \,\ell \ge 0$.
 \begin{Theo} \label{mainresult1} Let $k>\ell\ge 0$ and $f\in H_{\infty,\ell}^{\lambda}[\re > 0]$. If $f^{*}$ is continuous on some open interval $I\subset [\re = 0]$, then for all $i\tau \in I$
\begin{equation*}
\lim_{x\to \infty} R_{x}^{\lambda,k}(f)(i\tau)=f^{*}(i\tau)\,,
\end{equation*}
with uniform convergence on every closed sub interval $J\subset I$. Moreover, in this case $\big(R_{x}^{\lambda,k}(f)(\cdotp)\big)_{x>0}$ converges uniformly on each 'flattened cone'
\[
K(\gamma, J) = \big\{ z \in [\re >0]\colon z = iy + w \text{  with $iy \in J$ and $\arg(w) < \gamma$ }  \big\}\,,\,\,\, 0 < \gamma < \frac{\pi}{2},
\]
and for each $z = iy + w \in K(\gamma, J)$
\[
f(z) =
\lim_{x\to \infty}R_x^{\lambda,\ell}(f)(z)
=
\frac{w^{k+1}}{\Gamma(1+k)}\int_{0}^{\infty} t^k  R_{t}^{\lambda,k}(f)(iy) e^{-wt}  dt
\,.
\]
\end{Theo}

\begin{Rema}
 Indeed, a closer look at the proof of the 'moreover-part' shows, that we in fact prove  the following: Given a formal Dirichlet series $D=\sum a_{n}e^{-\lambda_{n}s}$, an interval $J\subset [\re = 0]$, and $0 < \gamma < \frac{\pi}{2}$, then $D $ is uniformly $(\lambda,k)$-Riesz summable on $K(\gamma, J)$, provided $D$ is uniformly $(\lambda,k)$-Riesz summable on $J$. Knowing this, the 'moreover-part' of Theorem~\ref{mainresult1}  is an immediate consequence of the first part.
\end{Rema}

Before we prove Theorem \ref{mainresult1} we illustrate this result  with two simple
examples.

The first example shows that the case $k=1$ and $\lambda = ( n)$ in fact is covered by
Fejer's theorem applied to  functions in $H_{\infty}(\T)\cap C(\T)$. In fact, given  $g \in C(\T)$, Fejer's theorem shows that uniformly for $z \in \T$
\begin{equation}\label{fejer}
  g(z)=\lim_{x\to \infty} \sum_{n<x} \hat{g}(n) z^n(1-\frac{n}{x})=\lim_{x\to \infty} \frac{1}{x}\sum_{n=1}^{x} \sum_{k=1}^{n}a_{k}z^k.
\end{equation}
For functions from $H_{\infty}(\T)\cap C(\T)$ this may be deduced from Theorem \ref{mainresult1}. Indeed, for each $g \in H_{\infty}(\T)\cap C(\T)$ there is a function $f \in H_{\infty}(\D)$
such that
$f^\ast = g$ on $\T$ and $\partial_n f(0)/n! = \hat{g}(n)$ for all $n \in \mathbb{N}_0$. But then as $g$ is continuous on $\T$ the outcome of Theorem~\ref{mainresult1}
with $k=1$ and $\lambda = ( n)$ precisely is
\eqref{fejer}, since the substitution $z = e^{-s}$ leads to a coefficient preserving isometry from
$H_{\infty}(\D)$ onto $H^{(n)}_{\infty,0}[\re > 0]$.

To see another example, denote by $\zeta:  \C \setminus \{1\} \to \C$ the zeta-function, which
is holomorphic with a simple pole in $s=1$, and which
on $[\re > 1]$ is the pointwise limit of the zeta-Dirichlet series $\sum n^{-s}$. Moreover, consider the entire function
\[
\eta: \C  \to \C\,, \,\,\, \eta(s) = (1 - 2^{1-s}) \zeta(s)\,,
\]
which on $[\re >0]$ is nothing else then the  pointwise limit of the $\eta$-Dirichlet series $\sum (-1)^{n+1} n^{-s}$.
As remarked in \cite[Section~3.1]{DefantSchoolmann6}
\begin{equation*} \label{ordereta}
\ell > \dfrac{1}{2}
\,\,\,\,\, \Rightarrow \,\,\,\,\,
\eta \in H^{(\log n)}_{\infty, \ell}[\re > 0]
\,\,\,\,\, \Rightarrow \,\,\,\,\,
\ell \ge \dfrac{1}{2}\,,
\end{equation*}
and in particular, $\eta\notin \mathcal{H}_{\infty}((\log n))$.
Hence, Theorem \ref{mainresult1} implies that for $k>1/2$
\begin{equation}\label{notoptimal}
\eta(it)=\lim_{x\to \infty} \sum_{\log(n)<x} (-1)^{n} n^{-it}(1-\frac{\log(n)}{x})^{k}
\end{equation}
uniformly on every closed interval $I\subset [\re = 0]$.

\begin{proof}[Proof of Theorem \ref{mainresult1}]
Without loss of generality
we all  over the proof assume that  that $\lambda_{1}=0$.

For the proof of the first part it (after translation)  suffices to check that $f^\ast$ is
$(\lambda,k)$-Riesz summable in $\tau=0$,  assuming  without loss of generality that $0\in I$. Moreover, we assume  that $f^\ast(0)=0$, since, if this is not the case, we may  consider $f-f^{*}(0) \in H_{\infty,\ell}^{\lambda}[\re > 0]$ instead of $f^\ast$.
As before we  distinguish the two cases $k <1$ and $k \ge 1$, and start with the first one.

 Fix $\varepsilon>0$ and choose $\delta>0$ such that $|f^\ast(iy)|\le \varepsilon$ for all $|y|\le \delta$, which is possible using the continuity of $f^{*}$ at the origin.
Then by Lemma \ref{Matea} and substitution
\begin{align*}
\Big|\int_{|y|\le \delta} f^{*}(iy)\frac{ R^{k,\ell}(x,y)}{(1+iy)^{\ell}} dy\Big|
&
\leq
\int_{|y|\le \delta} |f^{*}(iy) R^{k,\ell}(x,y)| dy
\\&
\le \varepsilon  \int_{|y|\le \delta } \frac{x}{|1+ixy|^{1+k-\ell}} dy
\le \varepsilon  \int_{\mathbb{R}} \frac{1}{|1+it|^{1+k-\ell}} dt.
\end{align*}
Hence,
 splitting like in \eqref{split} and using  Theorem \ref{perronformula22}, we finally get that
 \begin{equation*}
 \lim_{x\to \infty} R_{x}^{\lambda,k}(f)(0)=0=f^{*}(0).
 \end{equation*}
Since the second case $k \ge 1$ follows the same lines, the  first part of Theorem \ref{mainresult1} is accomplished.

For the proof of the second part assume that  $J\subset I$ is  a closed sub interval. Let $\delta_{0}>0$ be such that $J_{0}=J\pm i\delta_{0}\subset I$, and note that $f^{*}$ is uniformly continuous on $J_{0}$. Fix $\varepsilon>0$ and let $\delta_{0}>\delta>0$ such that $|f^{*}(i(y+\tau))-f^{*}(
\tau)|\le \varepsilon$ for all $|y|\le \delta$ and $i\tau \in J$.
Looking at \eqref{split}, using Theorem \ref{convolution1}, Corollary \ref{convolution1NEW}, and again Lemma \ref{Matea} as before,  for all $\varepsilon>0$ and $i\tau \in J$
\begin{align*}
&
| R_{x}^{\lambda,k}(f)(i\tau)-f^{*}(i\tau)|
  \\&
  \leq
\Big|\int_{|y|\ge \delta} f^\ast(i(y+\tau))\frac{R^{k,\ell}(x,y)}{(1+iy)^{\ell}}  dy\Big|
  +
  \Big|\int_{|y|\le \delta} (f^{*}(i(y+\tau))-f^{*}(i\tau))\frac{R^{k,\ell}(x,y)}{(1+iy)^{\ell}} dy \Big|
  \\&
  \leq \|f\|_{\infty,\ell}
  \int_{|y|\ge \delta} \frac{|(1+i(y+\tau)|^{\ell}}{|1+iy|^{\ell}} |R^{k,\ell}(x,y)|  dy
  + \varepsilon \int_{|y|\le \delta} |R^{k,\ell}(x,y)| dy
  \\&
    \le \|f\|_{\infty,\ell} C(\ell,J) \int_{|y|\ge \delta} |R^{k,\ell}(x,y)|  dy +\varepsilon C(k,\ell)\,,
  \end{align*}
  which then by  \eqref{sunday} (from Lemma \ref{Matea}) ensures that $\big(R_{x}^{\lambda,k}(f)(\cdotp)\big)_{x>0}$ converges uniformly to $f^\ast$ on $J$.

    To verify the 'moreover part' is slightly more involved. Choose $m\in~\N_{0}$ such that $m< k\le  m+1$. Fixing $iy \in J$, we define the $\lambda$-Dirichlet series
  \[
  D^y = -f^\ast(iy) + \sum a_n(f)e^{-\lambda_n s}
  \]
  (recall that
$\lambda_1=0$) and observe that for all  $x > 0$ and $s \in \mathbb{C}$
\begin{equation}\label{ibk}
  R_x^{\lambda,k}(D^y)(s) = -f^\ast(iy)+ R_x^{\lambda,k}(f)(s)\,.
\end{equation}
Applying  \cite[Lemma~4.11]{DefantSchoolmann6}  to $D$ (or more precisely to the horizontal  translation $D^y_{iy}$ of $D^y$ about $iy$), we see that there is a constant $L=L(m,k)$ such that for all
$z=\sigma+i\tau \in [\re >0]$ of the form $z= iy + w \in K(J,\gamma)$ (so  $iy \in J$ and $\arg(w) < \gamma$) we have
\begin{align*}
&\Big|\frac{x^{k}}{\Gamma(m+2)}\int_{0}^{x}t^{m+1}R_{t}^{\lambda,m+1}(D^y)(iy) w^{m+2}\Big(1- \frac{t}{x}\Big)^{k} e^{-wt} dt- x^{k}R_{x}^{\lambda,k}(D^y)(u)\Big|
\\[1ex] &
=\Big|\frac{x^{k}}{\Gamma(m+2)}\int_{0}^{x}t^{m+1}R_{t}^{\lambda,m+1}(D^y_{iy})(0) w^{m+2}\Big(1- \frac{t}{x}\Big)^{k} e^{-wt} dt- x^{k}R_{x}^{\lambda,k}(D^y_{iy})(w)\Big|
\\[1ex] &
\le L(m,k)  \sum_{j=1}^{m+1} |\sec(\gamma)|^{j} \int_{0}^{x} \big|t^{k}R_{t}^{\lambda,k}(D^y)(iy)\big| t^{-j}  (x-t)^{j-1}  dt \\ & + e^{-\sigma x} \big|x^{k}R_{x}^{\lambda,k}(D^y)(iy)\big|.
\end{align*}
From \cite[Lemma~4.12]{DefantSchoolmann6} we know that uniformly in $y$ and $w$ (as above)
\begin{align*}
\lim_{x\to \infty} \int_{0}^{x}t^{m+1}R_{t}^{\lambda,m+1}(D^y)(iy)& w^{m+2}\Big(1- \frac{t}{x}\Big)^{k} e^{-wt} dt
\\& =\int_{0}^{\infty}t^{m+1}R_{t}^{\lambda,m+1}(D^y)(iy) w^{m+2} e^{-wt} dt\,,
\end{align*}
and  by \cite[Lemma~4.5]{DefantSchoolmann6}
\begin{align*}
\int_{0}^{\infty}t^{m+1}R_{t}^{\lambda,m+1}(D^y)(iy) w^{m+2} e^{-wt} dt
=
\frac{\Gamma(m+2)}{\Gamma(k+1)}
\int_{0}^{\infty} t^{k}R_{t}^{\lambda,k}(D^y)(iy) w^{k+1} e^{-wt} dt\,,
\end{align*}
so together
\begin{align*}
\lim_{x\to \infty} \frac{1}{\Gamma(m+2)}\int_{0}^{x}t^{m+1}R_{t}^{\lambda,m+1}(D^y)&(iy) w^{m+2}\Big(1- \frac{t}{x}\Big)^{k} e^{-wt} dt
\\&
=
\frac{1}{\Gamma(k+1)}
\int_{0}^{\infty} t^{k}R_{t}^{\lambda,k}(D^y)(iy) w^{k+1} e^{-wt} dt\,.
\end{align*}
Moreover, following the proof of \cite[Theorem~2.9]{DefantSchoolmann6} and using \eqref{ibk}, we have that
\begin{align*} \label{duschen}
&
\lim_{x\to \infty}  x^{-k} \sum_{j=1}^{m+1} |\sec(\gamma)|^{j} \int_{0}^{x} \big|t^{k}R_{t}^{\lambda,k}(D^y)(iy)\big| t^{-j}  (x-t)^{j-1}  dt
\\&
\leq \lim_{x\to \infty}x^{-k} \sum_{j=1}^{m+1} |\sec(\gamma)|^{j} \int_{0}^{x} \sup_{iy\in J}|R_{t}^{\lambda,k}(f)(iy)-f^{*}(iy)|  t^{-(j-k)} (x-t
)^{j-1}  dt
=0\,,
\end{align*}
still uniformly in $y$.
All together we obtain that uniformly  in $y$ and $w$ (so uniformly in $z$)
  \begin{align*}
\lim_{x\to \infty}-f^\ast(iy)& + R_x^{\lambda,k}(f)(z)
\\&
=\frac{w^{k+1}}{\Gamma(1+k)}\int_{0}^{\infty} t^k \big(-f^\ast(iy)+ R_{t}^{\lambda,k}(f)(iy)\big) e^{-wt}  dt
\\&
=
-f^\ast(iy)
\frac{w^{k+1}}{\Gamma(1+k)}\int_{0}^{\infty} t^k  e^{-wt}  dt
+
\frac{w^{k+1}}{\Gamma(1+k)}\int_{0}^{\infty} t^k  R_{t}^{\lambda,k}(f)(iy) e^{-wt}  dt
\\&
=
-f^\ast(iy)
+
\frac{w^{k+1}}{\Gamma(1+k)}\int_{0}^{\infty} t^k  R_{t}^{\lambda,k}(f)(iy) e^{-wt}  dt
\,,
\end{align*}
and finally
\begin{align*}
\lim_{x\to \infty}R_x^{\lambda,k}(f)(z)
=
\frac{w^{k+1}}{\Gamma(1+k)}\int_{0}^{\infty} t^k  R_{t}^{\lambda,k}(f)(iy) e^{-wt}  dt
\,,
\end{align*}
which completes the argument.
\end{proof}

\section{A  Dini test} \label{DiniB}

Finally, we come to the last main contribution announced in the introduction -- a Dini test for functions in $H_{\infty,\ell}^{\lambda}[\re > 0]$.
\begin{Theo} \label{Dinitype1} Let $k>\ell\ge 0$ and $f\in H_{\infty,\ell}^{\lambda}[\re > 0]$. If for $\tau \in \R$
there is $\delta>0$  such that
\begin{equation*}
\int_{-\delta}^{\delta} \frac{|f^{*}(i(y+\tau))-f^{*}(i\tau)|}{|y|^{1+k-\ell}} dy<\infty\,,
\end{equation*}
then
\begin{equation*}
\lim_{x\to \infty} R_{x}^{\lambda,k}(f)(i\tau)=f^{*}(i\tau).
\end{equation*}
\end{Theo}

\begin{proof} As argued in the proof of Theorem \ref{mainresult1} we may assume that $\tau=0$ and $f^{*}(0)=0$ (provided that w.l.o.g. $\lambda_{1}=0$). According to  the splitting from \eqref{split} and Theorem \ref{perronformula22}, the claim follows once we show that
\begin{equation*}
\lim_{x\to\infty}\int_{|y|\le \delta} f^{*}(iy) \frac{R^{k,\ell}(x,y)}{(1+iy)^{\ell}} dy=0.
\end{equation*}
Indeed, by Lemma \ref{Matea}, provided $k<1$,
\begin{align*}
&\big|\int_{-\delta}^{\delta} f^{*}(iy) \frac{R^{k}(x,y)}{(1+iy)^{\ell}} dy\big| \le C(k,\ell)\int_{-\delta}^{\delta} |f^{*}(iy)| \frac{x}{|1+ixy|^{1+k-\ell}} dy\\ & =C(k,\ell)x^{-(k-\ell)} \int_{-\delta}^{\delta}  \frac{|f^{*}(iy)|}{|x^{-1}+iy|^{1+k-\ell}}dy\le C(k,\ell)x^{-(k-\ell)}  \int_{-\delta}^{\delta}  \frac{|f^{*}(iy)|}{|y|^{1+k-\ell}} dy,
\end{align*}
which by assumption vanishes as $x \to \infty$. The case $k\ge 1$ follows the same lines using Lemma \ref{Matea} accordingly.
\end{proof}

\section{A link to Carleson's theorem} \label{carleson}
Recall from \eqref{portugal} that for
every Dirichlet series $D = \sum a_n n^{-s}$ with $(a_{n})\in \ell_{2}$ (so in particular, if $D \in \mathcal{D}_{\infty}$) the so-called vertical limits
\begin{equation} \label{südtirol}
  \sum a_n \chi(n) n^{-it}
\end{equation}
converge for almost all characters
$\chi: \mathbb{N} \to \mathbb{T}$ and almost all $t \in \mathbb{R}$. In short, for such Dirichlet series almost all  vertical limits of $D$ converge almost everywhere on the boundary line.

In the introduction we indicate that this result in fact is a consequence of a Carleson-type convergence theorem for functions in $L_2(\mathbb{T}^\infty)$.
In \cite{defant2020variants} we proved a  Carleson-type  theorem for $\lambda$-Dirichlet series which belong to  the Hardy spaces $\mathcal{H}_p(\lambda), 1 < p \leq \infty$.
Inspired by \eqref{südtirol}, this result has consequences for the boundary behavior of almost all vertical limits of $\lambda$-Dirichlet series -- in particular if these series belong to the Banach space (see again \eqref{periodic-coin} and \eqref{l=0})
\begin{equation*}
\mathcal{H}_\infty(\lambda) = H_{\infty,0}^{\lambda}[Re>0] = H_{\infty}^{\lambda}[Re>0].
\end{equation*}
The aim of this subsection is to compare this output with what we now know about the boundary behaviour of functions in $H_{\infty}^{\lambda}[Re>0]$.
We start with a brief introduction to all
relevant notions.

\medskip

\noindent {\bf $\pmb{\lambda}$-Dirichlet groups.} Let $G$ be a compact abelian group and $\beta\colon (\R,+) \to G$ a homomorphism of groups.  Then the pair $(G,\beta)$ is called Dirichlet group, if $\beta$ is continuous and has dense range. In this case  the dual map $\widehat{\beta}\colon \widehat{G} \hookrightarrow \R$ is injective, where we identify $\mathbb{R}=\widehat{(\mathbb{R},+)}$ (note that we do not assume $\beta$ to be injective). Consequently, the  characters $e^{-ix\pmb{\cdot}} \colon \R \to \T$, $x\in \widehat{\beta}(\widehat{G})$, are precisely those which define  a unique $h_{x} \in \widehat{G}$  such that $h_{x} \circ \beta=e^{-ix\pmb{\cdot}}$. In particular, we have that
\begin{equation*}
\widehat{G}=\{h_{x} \mid x \in \widehat{\beta}(\widehat{G}) \}.
\end{equation*}
Now, given a frequency $\lambda$, we call a Dirichlet group $(G,\beta)$ a $\lambda$-Dirichlet group whenever $\lambda \subset \widehat{\beta}(\widehat{G})$, or equivalently whenever for every
$e^{-i\lambda_{n} \pmb{\cdot}} \in \widehat{(\mathbb{R},+)}$ there is (a unique) $h_{\lambda_{n}}\in  \widehat{G}$ with $h_{\lambda_{n}}\circ \beta=e^{-i\lambda_{n} \pmb{\cdot}}$.

For every $u>0$ the Poisson kernel
$P_{u}$
defines a measure $p_{u}$ on $G$, which we call the Poisson measure on $G$
(the push forward measure of $P_u dt$ under $\beta$).
We have $\|p_{u}\|=\|P_{u}\|_{L_{1}(\R)}=1$ and
\begin{equation*}
\text{$\widehat{p_{u}}(h_{x})=\widehat{P_{u}}(x)=e^{-u|x|}$ for all
$x\in \widehat{\beta}(\widehat{G})$.}
\end{equation*}
Finally, recall from \cite[Lemma 3.11]{defantschoolmann2019Hptheory} that, given  a measurable function $g:G \to \C$, then for almost all $\omega \in G$  there are measurable  functions $g_{\omega} \colon \R \to \C$
such that
\[
\text{$g_{\omega}(t)=g(\omega \beta(t))$ almost everywhere on $\R$,}
\]
and if $g\in L_{1}(G)$, then all these $g_\omega$ are locally integrable.

\medskip

\noindent {\bf Hardy spaces on $\pmb{\lambda}$-Dirichlet groups.}
Given a $\lambda$-Dirichlet group $(G,\beta)$ and $1\le p \le \infty$, by
 $H_{p}^{\lambda}(G)$
 we denote the  Hardy space of all functions
$g\in L_{p}(G)$ (recall that being a compact abelian group, $G$ allows a unique normalized Haar measure) having a Fourier transform  supported on $\{h_{\lambda_n} \colon n \in \mathbb{N}\} \subset \widehat{G}$. Being a closed subspace of $L_p(G)$, this clearly defines a Banach space.
 A fundamental fact from   \cite[Theorem 3.20]{defantschoolmann2019Hptheory} is that the definition of $H_{p}^{\lambda}(G)$ in the following sense is independent of the chosen $\lambda$-Dirichlet group $(G,\beta)$: If $(G_1,\beta_1)$ and  $(G_2,\beta_2)$
 are two $\lambda$-Dirichlet groups, then there is a Fourier coefficient preserving, isometric  and linear bijection
  identifying $H_p^\lambda(G_1)$ and $H_p^\lambda(G_1)$, i.e.
 \[
 H_p^\lambda(G_1) = H_p^\lambda(G_2)\,.
 \]
By $\mathcal{B}(f)=\sum \widehat{f}(h_{\lambda_{n}}) e^{-\lambda_{n}s}$ every $f\in H_{p}^{\lambda}(G)$ naturally generates  a $\lambda$-Dirichlet series, and the Hardy space $\mathcal{H}_{p}(\lambda)$ of $\lambda$-Dirichlet series is then defined to be the Banach space of all such Dirichlet series, i.e.
\begin{equation*}
\mathcal{H}_{p}(\lambda)=\{ D=\sum \widehat{f}(h_{\lambda_{n}}) e^{-\lambda_{n}s} \mid f\in H_{p}^{\lambda}(G)\},
\end{equation*}
together with the norm  $\|D\|_{p}=\|f\|_{p}$, whenever $D=\mathcal{B}(f)$.

The Carleson type theorem from \cite[Theorem~2.1]{defant2020variants} proves that, given $D=\sum a_{n}e^{-\lambda_{n}s}\in \mathcal{H}_{p}(\lambda)$ and a $\lambda$-Dirichlet group $(G,\beta)$, for almost every $\omega \in G$ the Dirichlet series $D^{\omega}=\sum a_{n} h_{\lambda_{n}}(\omega)e^{-\lambda_{n}s}$ (a so-called vertical limit of $D$) converges almost everywhere on the boundary line $[\re = 0]$, provided $p>1$.

\medskip

\noindent {\bf Examples.}
 Note that for every $\lambda$ there exists a $\lambda$-Dirichlet group $(G,\beta)$ (which is not unique).
To see a very first example, take the Bohr compactification $\overline{\R}$ together with the mapping
\begin{equation} \label{Bohr-comp}
\beta_{\overline{\R}} \colon \R \to \overline{\R}, ~~ t \mapsto \left[ x \mapsto e^{-itx} \right].
\end{equation}
Then $\beta_{\overline{\R}}$ is continuous and has dense range, and so the pair $(\overline{\R},\beta_{\overline{\R}})$ forms a $\lambda$-Dirichlet group for all $\lambda$'s. We refer to \cite{defantschoolmann2019Hptheory}  for more 'universal examples' of $\lambda$-Dirichlet groups. Looking at the  frequency  $\lambda=(n)=(0,1,2,\ldots)$, the group $G=\T$ together with \[\beta_\T: \R \to \T, \,\,\beta_{\T}(t)=e^{-it},\]
forms  a $\lambda$-Dirichlet group, and  the so-called
Kronecker flow
\begin{equation*}
\label{oscarHelson}
\beta_{\T^{\infty}}\colon \R \to \T^{\infty}, ~~ t \mapsto \mathfrak{p}^{-it}=(2^{-it},3^{-it}, 5^{-it}, \ldots),
\end{equation*}
turns the infinite dimensional torus $\T^{\infty}$ into a  $\lambda$-Dirichlet group
for $\lambda = (\log n)$.
We note that, identifying  $\widehat{\T} = \Z$ and $\widehat{\T^\infty} = \Z^{(\N)}$ (all finite sequences of integers), in the first case $h_n(z) = z^n$ for $z \in \T, n \in \Z$,
and in the second case $h_{\sum \alpha_j \log p_j}(z) = z^\alpha$ for  $z \in \T^\infty, \alpha \in \Z^{(\N)}$.

There is a useful reformulation of the Dirichlet group $(\T^\infty, \beta_{\T^{\infty}})$, which is already presented in the introduction. Denote by $\Xi$ the set of
    all characters $\chi: \N \to \T$, i.e. $\chi$ is completely multiplicative in the sense that
    $\chi(nm)=\chi(n)\chi(m)$ for all $n,m$. So every character is uniquely determined by its values on the primes.
    If we  on  $\Xi$ consider
   pointwise multiplication, then
\begin{align*} \label{idy}
\iota\colon \Xi \to \T^{\infty}, ~~ \chi \mapsto \chi(\mathfrak{p}) = (\chi(p_{n})),
\end{align*}
 is a  group isomorphism which turns  $\Xi$ into a compact abelian group. The Haar measure $d \chi$ on $\Xi$  is the push forward of the normalized Lebesgue measure $dz$ on $\T^\infty$   through $\iota^{-1}$. Hence also $\Xi$ together with
\begin{equation*}
\label{oscar2}
\beta_{\Xi}\colon \R \to \Xi, ~~ t \mapsto [p_k \to p_{k}^{-it}],
\end{equation*}
forms a $(\log n)$-Dirichlet group. Note that
$\mathbb{Z}^{(\mathbb{N})} =\widehat{\Xi}, \, \,\alpha  \mapsto \varphi $, where $\varphi (\chi) =\chi(\mathfrak{p})^\alpha$
for  $\chi \in \Xi$.

\medskip

\noindent {\bf Applying  Carleson's theorem.}
Fix some $f\in H_{\infty}^{\lambda}[\re > 0]$ and $\omega \in G$, where $(G,\beta)$ is a $\lambda$-Dirichlet group.
From \cite[Theorem~2.16]{defant2020riesz} (see also again \eqref{periodic-coin}) we know that there is an  isometric and coefficient preserving identity
\begin{equation}\label{johannahatbaldferien}
H_{\infty}^{\lambda}[\re > 0] =  \,H_{\infty}^\lambda(G)\,,\,\,\, f \mapsto g\,.
\end{equation}
Hence,
we deduce from \cite[Proposition~4.3]{defantschoolmann2019Hptheory} that there is a unique function
\begin{equation*}
f^\omega \in H_{\infty}^{\lambda}[\re > 0]
\end{equation*}
such that $a_n(f^\omega) = a_n(f) h_{\lambda_n}(\omega)$ for all $n$
and $\|f^\omega\|_\infty = \|f\|_\infty$. We call the function $f^\omega$   vertical limit of $f$ with respect to $\omega$ (and refer to
\cite[Proposition~4.6]{defantschoolmann2019Hptheory} which motivates this name).
Then by  Theorem~\ref{41} for each $k >0$ and $s \in [\re >0]$ the limit
\[
f^\omega(s) = \lim_{x \to \infty}\sum_{\lambda_n \leq x} a_n(f) h_{\lambda_n}(\omega) \Big( 1- \frac{\lambda_n}{x} \Big)^k
e^{-\lambda_n s}
\]
exist, i.e. $f^\omega$ is $(\lambda, k)$-Riesz summable on $[\re >0]$ for every $k >0$.

This  is in general not true for $k=0$ and all $\omega \in G$ (look at $\omega = e$, the unit in $G$, and some $\lambda$ not satisfying Bohr's theorem) and not true for $k=0$ and
all $s \in  [\re =0]$ (look at $\omega = e$, $\lambda = (\log n)$, and  the Bayart-Konyagin-Queff\'{e}lec example).

But an application of Carleson's theorem in the form given in \cite[Theorem~2.2]{defant2020variants} shows that    for each $f\in H_{\infty}^{\lambda}[\re > 0]$ the vertical
limits $f^\omega$ for almost all $\omega \in G$ are $(\lambda,0)$-Riesz-summable almost everywhere on the imaginary axis. Moreover, as we show now, if $g\in H_{\infty}^{\lambda}(G)$ is the function uniquely assigned  to $f$ in the sense of \eqref{johannahatbaldferien}, then  for almost all $\omega \in G$ the horizontal limit $(f^\omega)^\ast$ of the vertical limit
$f^\omega$ equals almost everywhere on $\mathbb{R}$ the 'restriction' $g_\omega(\tau) =  g(\omega \beta(\tau)),\, \tau  \in \mathbb{R}$.

\begin{Theo} \label{Helsoncase}  Let $f\in H_{\infty}^{\lambda}[\re > 0]$. Then for every $\lambda$-Dirichlet group  $(G,\beta)$, almost every $\tau\in \R$ and almost every $\omega \in G$
\[
\lim_{x\to \infty} \sum_{\lambda_n \leq x} a_n(f)h_{\lambda_n}(\omega)e^{-i\lambda_n \tau}=(f^\omega)^\ast(i\tau) \,.
\]
Moreover, if $g \in H_\infty^\lambda(G)$ is the unique function such that $a_n(f) =\widehat{g}(h_{\lambda_n}) $
for all $n$, then  for almost all $\tau\in \R$ and almost all $\omega \in G$
\[
g_\omega(\tau) = (f^\omega)^\ast(i\tau)\,.
\]
\end{Theo}

\begin{proof}
Let $g \in H_\infty^\lambda(G)$ be the unique function such that $a_n(f) =\widehat{g}(h_{\lambda_n}) $
for all $n$. Then $g \in H_2^\lambda(G)$, and by a variant of Carleson's  convergence theorem
  from \cite[Theorem~2.2]{defant2020variants} we know that
    \[
  g = \lim_{x\to \infty} \sum_{\lambda_n \leq x} a_n(f)h_{\lambda_n}\,\,\,\,\text{ almost everywhere on $G$} \,.
  \]
  Consequently, for almost all $\omega \in G$ by \cite[Lemma~1.4]{defant2020riesz} the limit
  \[
  g_\omega(\tau)=   \lim_{x\to \infty} \sum_{\lambda_n \leq x} a_n(f)h_{\lambda_n}(\omega)e^{-i \lambda_n \tau}
  \]
  exists for almost everywhere $\tau \in \mathbb{R}$. But by Corollary~\ref{existencehorizontalA}
  (used in the first equation of the following caculation) and \cite[Proposition 2.4]{defant2020riesz}
  (used in the fourth step to  change  limits), for almost all $\omega \in G$ and for almost all $\tau \in \mathbb{R}$ we have
  \begin{align*}
           (f^\omega)^\ast(i\tau)
     &
     =
     \lim_{\varepsilon \to 0} f^\omega( \varepsilon + i\tau)
     \\&
     =
     \lim_{\varepsilon \to 0}
     \lim_{x\to \infty} \sum_{\lambda_n \leq x} a_n(f)h_{\lambda_n}(\omega) \Big( 1- \frac{\lambda_n}{x} \Big)^k e^{-\varepsilon\lambda_n}e^{-i\lambda_n \tau}
     \\&
     =
     \lim_{\varepsilon \to 0}
     \lim_{x\to \infty} R_x^{\lambda,k} \big(g \ast  p_\varepsilon\big)  (\omega\beta(\tau))
     \\&
     =
          \lim_{x\to \infty}
          \lim_{\varepsilon \to 0}
          R_x^{\lambda,k} \big(g \ast  p_\varepsilon\big)  (\omega\beta(\tau))
          \\&
     =
          \lim_{x\to \infty} \sum_{\lambda_n \leq x} a_n(f)h_{\lambda_n}(\omega)\Big( 1- \frac{\lambda_n}{x} \Big)^k e^{-i\lambda_n \tau}
                              \\&
     =
          \lim_{x\to \infty} \sum_{\lambda_n \leq x} a_n(f)h_{\lambda_n}(\omega) e^{-i\lambda_n \tau}
               = g_\omega(\tau)\,,
  \end{align*}
 where the penultimate equation follows from the  fact that a $(\lambda,\ell)$-Riesz summable series is $(\lambda,k)$-Riesz summable
 for each $0\leq \ell \leq k$ with the same limit (see e.g \cite[Theorem 16, p. 29]{HardyRiesz}).  This completes the argument.
\end{proof}

We again believe that the ordinary case  is of independent interest.

\begin{Coro} \label{Helsoncase-ord}  Let $f\in H_{\infty}^{(\log n)}[\re > 0]$. Then for almost every $\tau\in \R$ and almost every $\chi \in \Xi$ we have
\[
\lim_{x\to \infty} \sum_{\lambda_n \leq x} a_n(f) \chi(n) n^{-i\tau }=(f^\chi)^{*}(i\tau)\,.
\]
Moreover, if $g \in H_\infty^{(\log n)}(\Xi)$ is the function  associated to $f$, i.e. $ \widehat{g}(\alpha) = a_n(f)$ for
 $n= \mathfrak{p}^\alpha$ and $ \widehat{g}(\alpha) = 0$ else, then  for almost all $\tau\in \R$ and almost all $\chi \in \Xi$
 \[
 (f^\chi)^{*}(i\tau) = g( n \mapsto \chi(n)n^{-i\tau})\,.
 \]
\end{Coro}

We finally illustrate  Theorem~\ref{Helsoncase} looking at bounded, holomorphic functions on the infinite dimensional
polydisc $B_{c_0}$. Take some $f\in H_{\infty}(B_{c_{0}})$. Then $\mathbb{T}^\infty$ may be seen as the 'distinguished
boundary' of $B_{c_0}$, and we may ask to which extent $f$ has boundary values.

We deduce as a consequence of Theorem~\ref{APcase}  (together with  \eqref{H=H}, \eqref{booky}, and \eqref{DD=HH}) that, given $k>0$, for almost every $t\in \R$
\begin{equation*} \label{ordinary1}
\lim_{\varepsilon\to 0}f\big(\mathfrak{p}^{-(\varepsilon +it)}\big)= \lim_{x\to \infty}  \sum_{\mathfrak{p}^\alpha < x} \frac{\partial^{\alpha}f(0)}{\alpha!} \Big( 1- \frac{\log \mathfrak{p}^\alpha}{x}  \Big)^k \frac{1}{\mathfrak{p}^{i\alpha t}}\,.
\end{equation*}

What can we in this case conclude from Theorem~\ref{Helsoncase}? To see this,
let $g \in H_{\infty}(\mathbb{T}^\infty)$
be associated to $f$ in the sense that $\widehat{g}(\alpha)= \frac{\partial^{\alpha}f(0)}{\alpha!}$ for $\alpha \in \mathbb{N}_0^{(\mathbb{N})}$, and $\widehat{g}(\alpha)= 0$ else.
Then by Theorem~\ref{Helsoncase}  for almost every $z\in \T^{\infty}$ and almost all $t \in \mathbb{R}$
\begin{equation*}\label{ordinary2}
g\big(z \mathfrak{p}^{-it}  \big)= \lim_{\varepsilon\to 0}f\big(z\mathfrak{p}^{-(\varepsilon +it)}\big)= \lim_{x\to \infty}  \sum_{\mathfrak{p}^\alpha < x} \frac{\partial^{\alpha}f(0)}{\alpha!}z^\alpha \frac{1}{\mathfrak{p}^{i\alpha t}}.
\end{equation*}

\section{Appendix} \label{Appendix}

\subsection{Elaborating Corollary \ref{convolution1NEW}} \label{AppendixB}
Recall from Corollary \ref{convolution1NEW}  that the mapping
\begin{equation*}
T \colon H_{\infty,\ell}^{\lambda}[\re > 0]\hookrightarrow L_{\infty}(\R), ~~f \mapsto\frac{f^\ast(i\cdot)}{(1+i\cdot)^{\ell}}
\end{equation*}
defines an isometric embedding. Starting with the following definition, we in  Theorem \ref{svwhsv} below  prove an internal description of the range of $T$:
Given a  frequency $\lambda$ and $\ell\ge 0$,
$$H_{\infty,\ell}^{\lambda}(\R)$$
 denotes  the subspace of all $g\in L_{\infty}(\R)$ for which  there are  a $\lambda$-Dirichlet series $D=\sum a_{n}e^{-\lambda_{n}s}$ and $m>0$ such that for every $u>0$
\begin{equation} \label{conditionB}
\lim_{x\to \infty} \sup_{t\in \R} \Big|\frac{(1+u+it)^{\ell}[g*P_{u}](t)}{(1+it)^{\ell}}-\frac{R_{x}^{\lambda,m}(D)(u+it)}{(1+it)^{\ell}}\Big| =0.
\end{equation}
For $\ell=0$ we write $$H_{\infty,0}^{\lambda}(\R)=H_{\infty}^{\lambda}(\R)\,.$$
The following result is an elaboration of  Corollary \ref{convolution1NEW}.

\begin{Theo} \label{svwhsv} Let $\ell\ge 0$ and $\lambda$ an arbitrary frequency. Then the mapping
\begin{equation*}
\Psi\colon H_{\infty,\ell}^{\lambda}(\R) \to H_{\infty,\ell}^{\lambda}[Re>0],~ \Psi(g)(u+it)=(1+u+it)^{\ell}[g*P_{u}](t).
\end{equation*}
defines a bijective isometry with inverse
\begin{equation*}
\Psi^{-1}\colon  H_{\infty,\ell}^{\lambda}[Re>0] \to H_{\infty,\ell}^{\lambda}(\R) ,~ \Psi(f)(t)=\frac{f^{*}(it)}{(1+it)^{\ell}},
\end{equation*}
where $f^{*}$ denotes the horizontal limit of $f$.
\end{Theo}

\begin{proof}
Fix $g\in H_{\infty,\ell}^{\lambda}(\R)$, and let $D=\sum a_{n}e^{-\lambda_{n}s}$ be some  Dirichlet series for which \eqref{conditionB} holds. Consider the function
\begin{equation*}
f(u+it)=(1+u+it)^{\ell}[g*P_{u}](t) \colon [Re>0] \to \C.
\end{equation*}
We claim that $f$ is holomorphic and that $D$ is the $\lambda$-Riesz germ of $f$. Indeed,  condition \eqref{conditionB} implies that for every $s\in [Re>0]$
\begin{equation} \label{convergenceAppendixB}
f(s)=\lim_{x\to \infty} R_{x}^{\lambda,m}(D)(s).
\end{equation}
Then the sequence $(R_{x}^{\lambda,m}(D))_x$ converges to $f$ uniformly on all compact subsets of $[Re>0]$ (see e.g. \cite[Theorem 2.9]{DefantSchoolmann6}), and so $f$ is holomorphic on $[Re>0]$ and  $D$ is the $\lambda$-Riesz germ of $f$.  Moreover,
\begin{align*}
\|f\|_{H_{\infty,\ell}^{\lambda}[Re>0]}=\sup_{u>0} \sup_{t\in \R} \Big| \frac{f(u+it)}{(1+u+it)^{\ell}} \Big| = \sup_{u>0} \|g*P_{u}\|_{L_{\infty}(\R)} \le \|g\|_{\infty},
\end{align*}
which eventually shows that $f\in H_{\infty,\ell}^{\lambda}[Re>0]$.
Let now $f\in H_{\infty,\ell}^{\lambda}[Re>0]$. Then by Corollary \ref{convolution1NEW} the function
$g = \frac{f^\ast(i\cdot)}{(1+i\cdot)^{\ell}}$ belongs to $L_{\infty}(\R)$ with $\|g\|_{\infty}=\|f\|_{\infty,\ell}$. It remains to show that $g\in H_{\infty,\ell}^{\lambda}(\R)$.  To do this, let  $D$ be the $\lambda$-Riesz germ of $f$, and note that by Theorem \ref{convolution1}
for every $u+it \in [\re > 0]$
\begin{equation}
(1+u+it)^{\ell}[g*P_{u}](t)=f(u+it).
\end{equation}
Then, by Theorem \ref{41},  \cite[Theorem 3.17]{DefantSchoolmann6} and continuity (of the function $f(u+\cdot)$ on $[Re\ge 0]$),  we for every $u>0$ and $k>\ell$ obtain that
\begin{align*}
0&=  \lim_{x\to \infty} \sup_{s\in [Re>0]} \Big | \frac{f(u+s)}{(1+s)^{\ell}}-\frac{R^{\lambda,k}_{x}(D)(u+s)}{(1+s)^{\ell}} \Big | \\ & = \lim_{x\to \infty} \lim_{v \to 0} \sup_{t\in \R} \Big | \frac{f(u+v+it)}{(1+v+it)^{\ell}}-\frac{R^{\lambda,k}_{x}(D)(u+v+it)}{(1+v+it)^{\ell}} \Big | \\ & \ge  \lim_{x\to \infty} \sup_{t\in \R} \Big | \frac{f(u+it)}{(1+it)^{\ell}}-\frac{R^{\lambda,k}_{x}(D)(u+it)}{(1+it)^{\ell}} \Big | \\ & =   \lim_{x\to \infty} \sup_{t\in \R} \Big | \frac{(1+u+it)^{\ell}}{(1+it)^{\ell}}[g*P_{u}](t)-\frac{R^{\lambda,k}_{x}(D)(u+it)}{(1+it)^{\ell}} \Big |,
\end{align*}
which proves that $g\in H_{\infty,\ell}^{\lambda}(\R)$.
Finally, we show that  almost everywhere on $\R$ we have
\begin{equation} \label{selenski}
\Psi(g)^{\ast}(t)=(1+it)^{\ell} g(t),
\end{equation}
and so $\Psi^{-1}(\Psi(g))=g$. Obviously  \eqref{selenski} is equivalent to the fact that almost everywhere on $\R$
\begin{equation} \label{standardLinfinityresult}
\lim_{u\to 0} [P_{u}*g](t)= g(t).
\end{equation}
But since   $g\in L_{\infty}(\R)$, this is standard, and we only for  the sake of completeness add an argument. Indeed, define $\alpha=g\chi_{[-N,N]}$ and $\beta=g-g\chi_{[-N,N]}.$ Then $\beta$ vanishes on $[-N,N]$, and so  we almost everywhere on $]-N,N[$ have
\begin{equation*}
\lim_{u\to 0} [P_{u}*\beta](t)=\beta(t)=0
\end{equation*}
(see e.g. \cite[Theorem 1.2.19, p. 27]{GrakakosClassic}). Additionally, since $\alpha\in L_{1}(\R)$,  almost everywhere on $]-N,N[$
\begin{equation*}
\lim_{u\to 0} [P_{u}*\alpha](t)=\alpha(t)=g(t).
\end{equation*}
(see e.g. Theorem \cite[Theorem 2.1.14., p. 94]{GrakakosClassic} and also  Example 2.1.15, p.95).
Together, almost everywhere on $]-N,N[$
\begin{equation*}
\lim_{u\to 0} [P_{u}*g](t)=\lim_{u\to 0} [P_{u}*\alpha](t)+[P_{u}*\beta](t)=g(t).
\end{equation*}
Now collecting countably many zero sets, proves \eqref{selenski}, which finishes the proof.
\end{proof}

The case $\ell=0$ is of special interest. In view of \eqref{l=0}  we may reformulate Theorem~\ref{svwhsv} as follows.

\begin{Coro} \label{svwhsvAA} Let  $\lambda$ be a frequency. Then the mapping
\begin{equation*}
\Psi\colon H_{\infty}^{\lambda}(\R) \to H_{\infty}^{\lambda}[Re>0],~ \Psi(g)(u+it)=[g*P_{u}](t).
\end{equation*}
is an bijective isometry with inverse
\begin{equation*}
\Psi^{-1}\colon  H_{\infty}^{\lambda}[Re>0] \to H_{\infty}^{\lambda}(\R) ,~ \Psi(f)(t)=f^{*}(it)\,.
\end{equation*}
\end{Coro}

Recall from Section~\ref{C} that the Banach space $H_{\infty}^{\lambda}[Re>0]$ is defined in terms of almost periodicity.
In the remaining part of this subsection we show how to give a similar  description for functions in  $H_{\infty}^{\lambda}(\R)$.

By $AP(\R)$ we denote the Banach space of all almost periodic functions (equipped with the sup-norm). If $g\in AP(\R)$, then the uniquely assigned Bohr coefficients of $g$ are given by
\begin{align} \label{coefBohr}
a_{x}(g)=\lim_{T\to \infty} \frac{1}{2T} \int_{-T}^{T} g(t)e^{ixt} dt, \quad x \in \mathbb{R}\,;
\end{align}
for functions in $H_{\infty}^{\lambda}[Re>0]$ compare this with  \eqref{bohrcoeffintro}.
Given a frequency $\lambda=(\lambda_{n})$, by $$AP^{\lambda}(\R)$$ we denote all functions $g\in AP(\R)$ for which the Bohr coefficients are supported on $\lambda$, i.e. the coefficient $a_{x}(g)$ vanishes, whenever $x \notin \{\lambda_{n} \mid n \in \N\}$.
We need a simple observation.

\begin{Lemm} \label{apholo}
Let $g\in AP^{\lambda}(\R)$. Then $g*P_{u}\in AP^{\lambda}(\R)$ for every $u>0$.
\end{Lemm}

\begin{proof}
Recall that all trigonometric polynomials are dense in $AP(\R)$ (see \cite[Theorem~1.5.5]{queffelec2013diophantine}).
  So, if $(p_n)$ is a sequence of trigonometric polynomials such that $\lim_{n \to \infty} p_{n }~=~g$ uniformly on $\mathbb{R}$, then also $\lim_{n \to \infty} p_{n}*P_{u} = g*P_{u}$ uniformly on $\mathbb{R}$, and so $g*P_{u}\in AP(\R)$ for every $u>0$.
  Checking on trigonometric polynomials and using density, shows
  \begin{equation}\label{lwiw}
    a_{x}(g*P_{u})=e^{-ux} a_{x}(g)\,, ~x\in \R,
  \end{equation}
  which completes the argument.
\end{proof}

The following result characterizes functions in $H_{\infty}^{\lambda}(\R)$ in terms of almost periodicity.

\begin{Theo} \label{casezero}
Let $\lambda$ be a frequency and $g\in L_{\infty}(\R)$. Then $g \in H_{\infty}^{\lambda}(\R)$ if and only if  $g*P_{u}\in AP^{\lambda}(\R)$ for every $u>0$.
Moreover, in this case
\[
\|g\|_\infty = \sup_{u >0} \| g*P_u  \|_\infty\,,
\]
and there is a $\lambda$-Dirichlet series such that for all $k >0$
and all  $u>0$
\begin{equation*}
\lim_{x\to \infty} \sup_{t\in \R} \Big|[g*P_{u}](t)-R_{x}^{\lambda,k}(D)(u+it)\Big| =0.
\end{equation*}

\end{Theo}

We prepare the proof with the following lemma.

\begin{Lemm} \label{apholoA}
The mapping
\begin{equation} \label{apmappingA}
\Phi: AP^{\lambda}(\R)\hookrightarrow H^{\lambda}_{\infty}[\re > 0],~ g \mapsto f,
\end{equation}
where
\[
\text{$f(u+it):=[g*P_{u}](it)$ for all $u+it \in [\re >0]$}\,,
\]
 defines an isometric and  coefficient preserving embedding.
\end{Lemm}
\begin{proof}
  Recall first that a function  on $\mathbb{R}$ is almost periodic if and only if it
  has a unique extension $\widetilde{g}$ to the Bohr compactification $\overline{\mathbb{R}}$
  (see again \cite[Theorem~1.5.5]{queffelec2013diophantine}). This in particular implies  that the mapping
  \[
  I: AP^\lambda(\mathbb{R}) \to C(\overline{\mathbb{R}}) \,, \quad g \mapsto \widetilde{g}
  \]
 is an isometric embedding (here it is used that $\beta_{\overline{\R}}$ from   \eqref{Bohr-comp} has dense range). Moreover, $I$ preserves Bohr- and Fourier coefficients, i.e. for all $x \in \mathbb{R}$
  \[
  a_{x}(g)  = \widehat{\widetilde{g}}(h_{x})
    \]
 (see e.g. \cite[Proposition~3.10]{defantschoolmann2019Hptheory}). So the range of $I$ is in fact contained in
 $H_{\infty}^\lambda(\overline{\mathbb{R}})$. On the other hand, we know from \eqref{johannahatbaldferien} that there is an  isometric, coefficient preserving bijection
\begin{equation*}
  J: H_{\infty}^\lambda(\overline{\mathbb{R}})  \to  H_{\infty}^{\lambda}[\re > 0] \,, \quad g \mapsto f\,.
\end{equation*}
Combining, we get an isometric, coefficient preserving embedding
\[
J\circ I: AP^{\lambda}(\R)\to H^{\lambda}_{\infty}[\re > 0]\,, \quad g \mapsto f\,,
\]
and then it remains to show that
\begin{equation} \label{minsk}
  [J\circ I](g)(u+it)=[g*P_{u}](it)\,, \quad u + it \in [\re > 0]\,.
\end{equation}
Indeed, fix some $u > 0$. Then for all $n$
\[
a_{\lambda_n} (g) = a_{\lambda_n} ([J\circ I]g)
=
\lim_{T \to \infty }  \frac{1}{2T}  \int_{-T}^{T}  \big([J\circ I]g\big)(u+it)e^{(u+it)\lambda_n }dt\,,
\]
(see again \eqref{bohrcoeffintro}), and hence by  Lemma~\ref{apholo}, \eqref{lwiw} and \eqref{coefBohr}
\[
a_{\lambda_n} (g \ast P_u) = a_{\lambda_n} (g) e^{-u \lambda_n} = a_{\lambda_n}(\big([J\circ I]g\big)(u+i\cdot))\,.
\]
Since almost periodic functions are uniquely determined  by their Bohr coefficients, \eqref{minsk} is proved.
    \end{proof}

\begin{proof}[Proof of Theorem \ref{casezero}]
If $g\in H_{\infty}^{\lambda}(\R)$, then $g*P_{u}$ by definition is the uniform limit of $\lambda$-Dirichlet polynomials, and so clearly $g*P_{u}\in AP^{\lambda}(\R)$ for every $u>0$.

Conversely, assume that $g \in L_\infty(\mathbb{R})$ such that $g*P_{u}\in AP^{\lambda}(\R)$ for every $u>0$. We define the bounded function $f(u+it):=(g*P_{u})(t)$, $u+it\in [Re>0]$, and claim that $f\in H_{\infty}^{\lambda}[Re>0]$ with $\Psi^{-1} f=g$, which finishes the proof by Corollary \ref{svwhsvAA}. To do this, we apply
Lemma~\ref{apholoA} to $h_n=g*P_{\frac{1}{n}}, \, n \in \mathbb{N}$, and  obtain that the function $f_n $ defined by
\[
f_n(u + it) = \big((g*P_{\frac{1}{n}}) *P_{u}\big)(t) = (g*P_{\frac{1}{n}+u})(t)\,,  \quad \text{ $u +it \in [\re>0]$}
\]
belongs to $H_{\infty}^{\lambda}[Re>0]$. Since for every $s\in [Re>\frac{1}{n}]$ and $n \in \mathbb{N}$
\begin{equation*}
f(s)=f(s-\frac{1}{n}+\frac{1}{n})=f_{n}(s-\frac{1}{n})\,,
\end{equation*}
 we see that $f$ is holomorphic on $[Re>0]$. Hence, being bounded and almost periodic on all vertical lines $[Re=\sigma]$, the function  $f$ belongs to $H_{\infty}^{\lambda}[Re>0]$. Moreover, the argument for  \eqref{standardLinfinityresult}  implies
 that for almost every $t\in \R$
\begin{align*}
  (\Psi^{-1} f)(t)
  &
  = f^\ast(it)
    = \lim_{u \to 0} f(u+it)
= \lim_{u \to 0}   (g*P_{u})(t)  = g(t),
\end{align*}
which finishes the proof of the  first claim. Finally, note that the first statement of the 'moreover part'
 then is also evident, whereas  the second follows from Theorem~\ref{41} looking at the $\lambda$-Riesz germ $D$ of $f$.
\end{proof}

In view of Theorem \ref{Helsoncase} and Corollary \ref{svwhsvAA} we add the following observation.
\begin{Coro}Let $(G,\beta)$ be a $\lambda$-Dirichlet group and $g\in H_{\infty}^{\lambda}(G)$. Then $g_{\omega}\in H_{\infty}^{\lambda}(\R)$ for almost every $\omega \in G$.
\end{Coro}
\begin{proof}
For almost every $\omega \in G$ we know from \cite[Lemma 3.11]{defantschoolmann2019Hptheory} that $g_{\omega} \in L_{\infty}(\R)$, and moreover by Theorem \ref{Helsoncase}, that $g_{\omega}$ is the horizontal limit of a function $f^{\omega} \in H_{\infty}^{\lambda}[Re>0]$. Hence by Corollary \ref{svwhsvAA}, we obtain $g_{\omega} \in H_{\infty}^{\lambda}(\R)$ for almost every $\omega \in G$.
\end{proof}
We complete this section by another observation, which states that in the ordinary case $H_{\infty}^{(\log n)}(\R)$ may be describe in terms of Ces\`{a}ro limits.
\begin{Coro}A function $g\in L_{\infty}(\R)$ belongs to $H_{\infty}^{(\log n)}(\R)$ if only only if there is an ordinary Dirichlet series $D=\sum a_{n}n^{-s}$ such that for every $u>0$ the Ces\`{a}ro means of $D$ converges uniformly on $[Re=u]$ to $g*P_{u}$, i.e. for every $u>0$
\begin{equation} \label{cesaro}
\lim_{x\to \infty} \sup_{t\in \R} \Big| g*P_{u}(t)-\frac{1}{x}\sum_{y\le x} \sum_{n\le y} a_{n}n^{-(u+it)} \Big|=0.
\end{equation}
\end{Coro}
\begin{proof}
Observe that by Theorem \ref{casezero}  condition \eqref{cesaro} immediately implies that $g\in H_{\infty}^{\lambda}(\R)$. Conversely, if $g\in H_{\infty}^{\lambda}(\R)$, then Theorem \ref{casezero} implies that there is an ordinary Dirichlet series $D$ such that the $((\log n),k)$-Riesz means of $D$ of any order $k>0$ on $[Re=u]$ converges uniformly to $g*P_{u}$ for every $u>0$. Now using
\cite[Theorems~17~and~30]{HardyRiesz} (see also
 \cite[Theorem 2.7, ii)]{DefantSchoolmann6} and \cite[Theorem 2.8]{DefantSchoolmann6}),  and repeating the vector-valued argument ,of \cite[Corollary~2.17]{DefantSchoolmann6} for fixed $u>0$, we obtain that the $(n,k)$-Riesz means of $D$ of any order $k>0$ converge uniformly on $[Re>2u]$ with limit $g*P_{2u}$ for every $u>0$. Since the $((n),1)$-Riesz means of $D$  are precisely the Ces\`{a}ro means of $D$, the proof is complete.
\end{proof}

\smallskip

\subsection{A Proof of Riesz' Theorem \ref{Theo42}} \label{Appendix2}

A crucial ingredient of our proof of Theorem \ref{Theo42} is given by the Perron-type formula from \cite[Theorem~2.13]{DefantSchoolmann6}:

For $k \geq 0$ let  $D=\sum a_{n} e^{-\lambda_{n}s}$ be a somewhere $(\lambda,k)$-Riesz summable $\lambda$-Dirichlet series  and
 $f: [ \re >\sigma_c^{\lambda,k}(D)] \to \mathbb{C}$ its limit function, where $\sigma_c^{\lambda,k}(D)$ stands for the abscissa of $(\lambda,k)$-Riesz summability of $D$. Then
\begin{equation} \label{Perronsformula2}
R_{x}^{\lambda,k}(f)(0)=\frac{\Gamma(1+k)}{2\pi i}  x^{-k} \int_{c-i\infty}^{c+i\infty} \frac{f(s)}{ s^{1+k}} e^{xs} ds,~~c>0.
\end{equation}
 Given $s \in \mathbb{C}$ and $r >0$, the euklidean ball in $\mathbb{C}$ of radius $r$ and center $s$ is denoted by $B_{r}(s)$.

\begin{proof}[Proof of Theorem \ref{Theo42}]
After translation we may assume that the poles of $f$ are given by $0<|p_{1}|\le |p_{2}|\le \ldots \le |p_{N}|$ with orders $m_{j}=m(p_{j})<1+k$.  We claim that for  every $0 < 2\delta < |p_1|$ and $I = [- i\delta,i\delta]$
\begin{equation}\label{poles}
\lim_{x\to \infty} \sup_{i\tau \in I} |R_{x}^{\lambda,k}(f)(i\tau)-f(i\tau)|=0.
\end{equation}
Indeed, provided that this claim is established, take an arbitrary interval
\begin{equation*}
I=i[a,b]\subset [\re = 0]\setminus \{p_{1},\ldots, p_{N}\}\,.
\end{equation*}
Then for every $i\tau\in I$ the translation $f_{i\tau}(s)=f(s+i\tau)$ of $f$ about $i\tau$ is uniformly $(\lambda,k)$-Riesz summable on $i[-\delta(\tau), \delta(\tau)]$ for some $\delta(\tau) >0$, and so $f$ is uniformly $(\lambda,k)$-Riesz summable on
$[-i\delta(\tau) +i\tau, i\delta(\tau) +i\tau]$. Since
\begin{equation*}
I\subset \bigcup_{i\tau \in I}]-i\delta(\tau) +i\tau, i\delta(\tau) +i\tau[,
\end{equation*}
by compactness there are finitely many $\tau_{1}, \ldots, \tau_{K}\in I$ such that
\begin{equation*}
I\subset \bigcup_{j=1}^{K} [-i\delta(\tau_j) +i\tau_j, i\delta(\tau_j) +i\tau_j]\,,
\end{equation*}
and consequently  $f$ is uniformly $(\lambda,k)$-Riesz summable on $I$.

Let us start the proof of \eqref{poles}, fixing some $0 < 2\delta < |p_1|$. The choice $c=x^{-1}$ in~\eqref{Perronsformula2} leads to
\begin{equation*}
R_{x}^{\lambda,k}(f)(i\tau) =\frac{\Gamma(1+k)e}{2\pi i} \int_{-\infty}^{\infty} f(x^{-1}+i(t+\tau))e^{itx} \frac{x}{(1+ixt)^{1+k}} dt.
\end{equation*}
The idea now is to split this integral  with respect to a disjoint union of sub intervals of $\mathbb{R}$. To do so, choose some $\varepsilon >0$
such that
\begin{equation} \label{isolatedpoles}
\bigcap_{j=1}^N\,\, ]p_j -2\varepsilon, p_j +2\varepsilon[    = \emptyset
\end{equation}
and consider the disjoint decomposition
\[
\R=S\cup [\R\setminus S],~ \text{where}~S=]-\delta, \delta[\,\,\cup \,\,\bigcup_{j=1}^{N} \,]p_{j}-\varepsilon,p_{j}+\varepsilon[\,.
\]
Observe that  $\R\setminus S$ is the union of finitely many  disjoint intervals $J$ formed by the connected components of  $\R\setminus S$.
Now  we  show first that
\begin{equation} \label{red1}
\lim_{x \to \infty} \sup_{\tau \in I}\Big|\int_{-\delta}^{\delta} f(x^{-1}+i(t+\tau))e^{itx} \frac{x}{(1+ixt)^{1+k}} dt-f(i\tau)\Big|=0,
\end{equation}
then second that for all $1\le j \le N$
\begin{equation} \label{red3}
\lim_{x\to \infty} \sup_{\tau \in I}\Big|\int_{p_{j}-\varepsilon}^{p_{j}+\varepsilon} f(x^{-1}+i(t+\tau)) e^{ixt}\frac{x}{(1+ixt)^{1+k}} dt\Big|=0,
\end{equation}
and finally that for all connected components $J$ of $\mathbb{R} \setminus S$
\begin{equation} \label{red2}
\lim_{x \to \infty}  \sup_{\tau\in I}\Big|\int_{J} f(x^{-1}+i(t+\tau))e^{itx} \frac{x}{(1+ixt)^{1+k}} dt\Big|=0\,.
\end{equation}
Note that the proof is complete, whenever these three claims are provided.

Proof of \eqref{red1}: By substitution for every $i\tau \in I$ and $x >1$
\begin{align*}
&\int_{-\delta}^{\delta} f(x^{-1}+i(t+\tau))e^{itx} \frac{x}{(1+ixt)^{1+k}} dt\\&=\int_{-x\delta}^{x\delta} f(x^{-1}+i(yx^{-1}+\tau))e^{iy} \frac{1}{(1+iy)^{1+k}} dy\\ &= \int_{\mathbb{R}} \chi_{x[-\delta,\delta]}(y) f(x^{-1}+i(yx^{-1}+\tau))e^{iy} \frac{1}{(1+iy)^{1+k}} dy.
\end{align*}
Moreover,
\begin{align*}
|\chi_{x[-\delta,\delta]}&(y) f(x^{-1}+i(yx^{-1}+\tau))e^{iy} (1+iy)^{-(1+k)}|\le
\sup_{\substack{|s|\leq 1 +2\delta\\ \re s> 0 } }|f(s)|
|1+iy|^{-(1+k)}\,,
\end{align*}
since $|x^{-1}+i(yx^{-1}+\tau)| \leq 1 + 2 \delta$ for all $y \in x[-\delta, \delta]$\,.
Additionally, since $f$ is uniformly continuous on $[0,1] \times 2I$, we for every $y\in \R$ have uniformly for $i\tau \in I$
\begin{equation*}
\lim_{x\to \infty} \chi_{x[-\delta,\delta]}(y) f(x^{-1}+i(yx^{-1}+\tau))e^{iy} (1+iy)^{-(1+k)}=f(i\tau)e^{iy} |1+iy|^{-(1+k)}.
\end{equation*}
Hence the (uniform) dominated convergence theorem together with  \eqref{Fouriervalue} shows that  uniformly for~$\tau~\in~I$
\begin{equation*}
\lim_{x \to \infty} \int_{-\delta}^{\delta} f(x^{-1}+i(t+\tau))e^{itx} \frac{x}{(1+ixt)^{1+k}} dt=f(i\tau)
\int_{-\infty}^{\infty}e^{iy} \frac{1}{(1+iy)^{1+k}} dy = f(i\tau)\,.
\end{equation*}

Proof of \eqref{red3}: Fix $1 \leq j \leq N$ and let $x>1$.
Since by assumption and (\ref{isolatedpoles})
\begin{equation*}
\sup_{s\in B_{\varepsilon}(ip_{j})\cap[Re> 0]}|(s-ip_{j})^{m_{j}}f(s)|= C_{j}<\infty,
\end{equation*}
we for every $i\tau \in I$ conclude that
\begin{align*}
&\Big|\int_{p_{j}-\varepsilon}^{p_{j}+\varepsilon} f(x^{-1}+i(t+\tau)) e^{ixt}\frac{x}{(1+ixt)^{1+k}} dt\Big|\\ &\le C_{j}\int_{p_{j}-\varepsilon}^{p_{j}+\varepsilon} \frac{1}{|x^{-1}+i(t+\tau-p_{j})|^{m_{j}}} \frac{x}{|1+ixt|^{1+k}} dt\\ &\le C_{j} \int_{-\varepsilon}^{\varepsilon} \frac{1}{|x^{-1}+i(y+\tau)|^{m_{j}}} \frac{x}{|1+ix(y+p_{j})|^{1+k}} dy \\ &
\le C_{j}x^{-k} p_{j}^{-(1+k)} \int_{-\varepsilon}^{\varepsilon} \frac{1}{|x^{-1}+i(y+\tau)|^{m_{j}}} dy\,,
\end{align*}
where, since $x>1$, for the last estimate we use that $  p_j \leq |1+ix(y+p_{j})| $.
Moreover,
\begin{align*}
\int_{-\varepsilon}^{\varepsilon} \frac{1}{|x^{-1}+i(y+\tau)|^{m_{j}}} dy
&
=x^{m_{j}-1} \int_{-\varepsilon+\tau}^{\varepsilon+\tau} \frac{x}{|1+ixy|^{m_{j}}}dy
\\&
\le x^{m_j-1} \int_{(-\varepsilon+\tau)x}^{(\varepsilon+\tau)x} \frac{1}{|1+ir|^{m_{j}}} dr.
\end{align*}
If $m_{j}\ge 2$, then
\begin{equation*}
\int_{-\infty}^{\infty} \frac{1}{|1+ir|^{m_{j}}} dr< \infty
\end{equation*}
and if $ m_j=1$, there is  $C=C(I,\varepsilon)>0$ such that for all~$x,~\tau$
\begin{align*}
&\int_{(-\varepsilon+\tau)x}^{(\varepsilon+\tau)x} \frac{1}{|1+ir|} dr\le \int_{-Cx}^{Cx}\frac{1}{|1+ir|} dr=2\int_{0}^{Cx}\frac{1}{|1+ir|} dr \\ &\le 2\big(1 +\int_{1}^{Cx} \frac{1}{|1+ir|} dr\big)\le 2+ \int_{1}^{C} \frac{1}{r} dr =2+\ln(Cx).
\end{align*}
Hence  all in all we obtain  some $D=D(I,m_j,\varepsilon)>0$  such that for all $x, \tau$
\begin{equation*}
\sup_{\tau \in I}\Big|\int_{p_{j}-\varepsilon}^{p_{j}+\varepsilon} f(x^{-1}+i(t+\tau)) e^{ixt}\frac{x}{(1+ixt)^{1+k}} dt\Big|\le D p_{j}^{-(1+k)}\ln(x)x^{m_{j}-(1+k)},
\end{equation*}
which vanishes as $x\to \infty$, since $m_{j}<k+1$.

\smallskip
Proof of \eqref{red2}: Note first that  each of the finitely many  connected components  $J$ of    $\mathbb{R} \setminus S$ is an interval, and that all of them except two are bounded. Fix such interval $J = i [a,b]$. Then, using in the bounded case the continuity  of $f$ on $[0,1] + J$ and in the unbounded case moreover
the assumption made on  the growth of $f$, we for all $\tau \in I$ and $t\in J$ have
\begin{equation*}
|f(x^{-1}+i(t+\tau))|\le  C(J) |x^{-1}+i(t+\tau)|^{\ell}\le C(J,I) |1+itx|^{\ell}\,.
\end{equation*}
 Consequently
\begin{align*}
\Big|\int_{J} f(x^{-1}+i(t+\tau))e^{itx} \frac{x}{(1+ixt)^{1+k}} dt\Big|
&
\le C(J,I) \int_{a}^{b} \frac{x}{|1+itx|^{1+k-\ell}} dt \\ & =C(I,J) \int_{ax}^{bx} \frac{1}{|1+iy|^{1+k-\ell}} dy,
\end{align*}
which vanishes as $x\to \infty$.
\end{proof}

\end{document}